\documentclass[reqno]{amsart}
\usepackage{fullpage,amsmath, amssymb,amscd}
\usepackage[mathscr]{eucal}
\usepackage{xypic,graphicx}
\usepackage{amscd}
\usepackage{mathrsfs}
\usepackage[all, cmtip]{xy}

\newcommand\cyr{%
 \renewcommand\rmdefault{wncyr}%
 \renewcommand\sfdefault{wncyss}%
 \renewcommand\encodingdefault{OT2}%
\normalfont\selectfont} \DeclareTextFontCommand{\textcyr}{\cyr}

\newtheorem{theorem}{Theorem}
\newtheorem{lemma}[theorem]{Lemma}
\newtheorem{corollary}[theorem]{Corollary}
\newtheorem{proposition}[theorem]{Proposition}
\newtheorem{remark}[theorem]{Remark}

\newcommand{\begd}{\begin{displaystyle}}

\newcommand{\ord}{\textrm{ord}}
\newcommand{\sgn}{\textrm{sgn}}

\newcommand{\cal}[1]{\mathcal{#1}}

\newcommand{\OO}{\mathcal{O}}
\newcommand{\MM}{\mathcal{M}}
\newcommand{\fp}{\mathfrak{p}}

\newcommand{\fl}{\mathfrak{l}}

\newcommand{\newd}{\mathscr{D}}

\def\Z{\mathbb Z}
\def\N{\mathbb N}
\def\Q{\mathbb Q}
\def\R{\mathbb R}
\def\C{\mathbb C}
\def\F{\mathbb F}

\def\G{\mathbb G}

\def\PP{\mathcal{P}}

\def\Ker{\operatorname{Ker}}
\def\Im{\operatorname{Im}}

\def\deg{\operatorname{deg}}
\def\det{\operatorname{det}}

\def\Drin{\operatorname{Drin}}
\def\End{\operatorname{End}}

\def\Hom{\operatorname{Hom}}

\def\Aut{\operatorname{Aut}}
\def\Gal{\operatorname{Gal}}

\def\Centr{\operatorname{Centr}}

\def\Quot{\operatorname{Quot}}

\def\mod{\operatorname{mod}}

\def\lcm{\operatorname{lcm}}
\def\gcd{\operatorname{gcd}}

\def\GL{\operatorname{GL}}

\def\O{\operatorname{O}}
\def\o{\operatorname{o}}

\def\log{\operatorname{log}}

\def\ord{\operatorname{ord}}

\def\rad{\operatorname{rad}}

\def\ds{\displaystyle}

\begin{document}

\title{
The distribution of the first elementary divisor of the reductions of a generic Drinfeld module of arbitrary rank
}


\author{
Alina Carmen Cojocaru and Andrew Michael Shulman }
\address[Alina Carmen  Cojocaru]{
\begin{itemize}
\item[-]
Department of Mathematics, Statistics and Computer Science, University of Illinois at Chicago, 851 S Morgan St, 322
SEO, Chicago, 60607, IL, USA;
\item[-]
Institute of Mathematics  ``Simion Stoilow'' of the Romanian Academy, 21 Calea Grivitei St, Bucharest, 010702,
Sector 1, Romania
\end{itemize}
} \email[Alina-Carmen  Cojocaru]{cojocaru@uic.edu}
\address[Andrew  M. Shulman]{
\begin{itemize}
\item[-]
Department of Mathematics, Statistics and Computer Science, University of Illinois at Chicago, 851 S Morgan St, 322
SEO, Chicago, 60607, IL, USA;
\end{itemize}
} \email[Andrew M. Shulman]{ashulm2@uic.edu}

\thanks{
A.C. Cojocaru's  work on this material was partially supported by the National Science Foundation under agreements No.
DMS-0747724 and No. DMS-0635607, and by the European Research Council 
under Starting Grant 258713.
}

\maketitle

\begin{abstract}\noindent
Let $\psi$ be a generic Drinfeld module of rank $r \geq 2$. We study the first elementary divisor 
$d_{1, \wp}(\psi)$ of the reduction of $\psi$ modulo a prime $\wp$, as $\wp$ varies.
In particular, we prove the existence of  the density of the primes $\wp$ for which $d_{1, \wp} (\psi)$ is fixed. For $r = 2$, we also study the second elementary divisor (the exponent) of the reduction of $\psi$ modulo $\wp$
and prove that, on average,  it has a large norm. Our work is motivated by the study of J.-P. Serre of an elliptic curve analogue of Artin's Primitive Root Conjecture,  and,  moreover,  by refinements to Serre's  study developed by  the first author and M.R. Murty. 
\end{abstract}


\section{Introduction and statement of results}

A beautiful and fruitful theme in number theory is that of exploring versions of one given problem in both the number
field and function field settings. In many instances, such explorations unravel striking analogies, shedding light to
deep basic principles underlying the problem. In other instances, the number field and function field versions of the
same problem turn out to be surprisingly different.

This article is part of such dual investigations, where the problem is that of {\it{Frobenius distributions in
GL-extensions,}} generated by elliptic curves over number fields and by Drinfeld modules over function fields. In
particular, the article focuses on the problem of determining 
the \emph{distribution of the first elementary divisor}
of the reduction modulo a prime of a generic Drinfeld module, as the prime varies. 
Our main result is analogous to 
a generalization of a result of J.-P. Serre \cite{Se}, proven in \cite{CoMu} and \cite{Co3}, for the reductions modulo primes of an elliptic curve over $\Q$.
The techniques used in proving our main result lead to further applications, such as to Drinfeld module analogues of a result of  W. Duke \cite{Du} and of a recent result of T. Freiberg and P. Kurlberg \cite{FrKu},
as we now explain.

Let $E/\Q$ be an elliptic curve over $\Q$, and for a prime $p$ of good reduction, let $E_p/\F_p$ be the reduction of
$E$ modulo $p$. By the theory of torsion points of elliptic curves, there exist uniquely determined
positive integers $d_{1, p}(E), d_{2, p}(E)$ such that
\begin{equation*}\label{elem-div-ec}
E_p(\F_p) \simeq_{\Z} \Z/d_{1, p}(E) \Z \times \Z/d_{2, p}(E) \Z
\end{equation*}
and
\begin{equation*}
d_{1, p}(E) | d_{2, p}(E).
\end{equation*}
In the theory of $\Z$-modules, the integers $d_{1, p}(E), d_{2, p}(E)$ are called the {\bf{elementary divisors}} of
$E_p(\F_p)$, with the largest of them, $d_2 = d_{2, p}(E)$, called the {\bf{exponent}}, having the property that $d_2 x = 0$ for all $x \in E_p(\F_p)$   (see the general definition in \cite[p. 149]{La}).

The study of the growth of $d_{2, p}(E)$, as the prime $p$ varies and $E/\Q$ is fixed, was initiated by R. Schoof
\cite{Sc}, who showed that, if $\End_{\overline{\Q}}(E) \simeq \Z$, then
\begin{equation} \label{schoof}
d_{2, p}(E) \gg \frac{\log p}{(\log \log p)^2} \sqrt{p}.
\end{equation}
W. Duke \cite{Du} improved this bound substantially, but in an ``almost all'' sense. To be precise, Duke showed that,
given any positive function $f$ with $\ds\lim_{x \rightarrow \infty} f(x) = \infty$, 
then, as $x \rightarrow \infty$,
\begin{equation}\label{duke}
\#\left\{p \leq x: d_{2, p}(E) > \frac{p}{f(p)}\right\} \sim \pi(x),
\end{equation}
unconditionally if $\End_{\overline{\Q}}(E) \not\simeq \Z$, and conditionally upon the Generalized Riemann Hypothesis (GRH) if $\End_{\overline{\Q}}(E)
\simeq \Z$; here,  $\pi(x)$ denotes the number of primes $ p \leq x$. 
By the ``Riemann hypothesis for curves over finite fields'' (Hasse's bound in this case),
 the numerator $p$ in the growth $\frac{p}{f(p)}$ of $d_{2, p}(E)$ above is nothing but the order of magnitude of $\#E_p(\F_p)$. Thus, roughly, Duke's result says that, for almost all $p$, the exponent of 
 $E_p(\F_p)$ is almost as large as the order of $E_p(\F_p)$. 
 This behaviour is also confirmed by a recent result  of T. Freiberg and P. Kurlberg 
 \cite{FrKu} (see also the follow up papers by S. Kim \cite{Ki} and J. Wu \cite{Wu}), in the following sense. 
 Under the same assumptions as Duke's, Freiberg and Kurlberg  showed that,
 as $x \rightarrow \infty$,
  \begin{equation}\label{freiberg-kurlberg}
 \frac{1}{\pi(x)} \ds\sum_{p \leq x} d_{2, p}(E) \sim  c(E) x
 \end{equation}
for some explicit constant $c(E) \in (0, 1)$, depending  on $E$. 

The proofs of (\ref{duke}) and (\ref{freiberg-kurlberg}) reduce to the analysis of sums of
the form 
$$\ds\sum_{y < d < z} \#\{p \leq x:  d | d_{1, p}(E)\}$$ 
for  suitable parameters $y = y(x), z = z(x)$.
In particular, they reduce to an understanding of the first elementary divisor  $d_{1, p}(E)$.

The study of $d_{1, p}(E)$, as the prime $p$ varies and $E/\Q$ is fixed, has been carried out  for over four decades and precedes the study of $d_{2, p}(E)$. Most notably, J.-P. Serre \cite{Se} studied the distribution of the primes
$p$ for which $d_{1, p}(E) = 1$  in analogy to the study of the Artin primitive root conjecture, while M. R. Murty \cite{Mu} and, later, the first author of this paper,  refined and strengthened Serre's result, proving the following (see \cite{Co1}, \cite{Co2},  \cite{CoMu}, and \cite{Co3}): for any $d \in \N$, there exists an explicit constant $\delta_{E, \Q}(d) \geq 0$ such that, 
as $x \rightarrow \infty$,
\begin{equation}\label{cojocaru-murty}
\#\{p \leq x: d_{1, p}(E) = d\} \sim \delta_{E, \Q}(d) \pi(x),
\end{equation}
unconditionally if $\End_{\overline{\Q}}(E) \not\simeq \Z$, and conditionally upon GRH if $\End_{\overline{\Q}}(E)
\simeq \Z$.  Under GRH, Cojocaru and Murty \cite{CoMu} (see also \cite{Co3}) showed that the error term in this asymptotic is
$\O_{E, d}\left(x^{\frac{3}{4}} (\log x)^{\frac{1}{2}}\right)$ if $\End_{\overline{\Q}}(E) \not\simeq \Z$,
and
 $\O_{E, d}\left(x^{\frac{5}{6}} (\log x)^{\frac{2}{3}}\right)$ if $\End_{\overline{\Q}}(E) \simeq \Z$.

When considering the function field analogue of these problems, we are naturally led to Drinfeld modules. Indeed, the
role played by {\it{elliptic curves}} over $\Q$ in number field arithmetic is similar to the one played by {\it{rank 2
Drinfeld modules}} over $\F_q(T)$ in function field arithmetic. Drinfeld modules  also come in higher generalities, for
example in higher ranks,
and, as such,  when suitable,
we may focus on Drinfeld modules of {\it{arbitrary}} rank.

To state our main results, we fix the following: $q$ a prime power; $A := \F_q[T]$; $k := \F_q(T)$; $K \supseteq k$ a
finite field extension; $\psi: A \longrightarrow K\{\tau\}$ a generic Drinfeld $A$-module over $K$, of rank $r \geq 2$. Here, $\tau: x
\mapsto x^q$ is the $q$-th power Frobenius  automorphism and $K\{\tau\}$ is the skew-symmetric polynomial ring in $\tau$ over $K$ (we will review
definitions and basic properties in Sections 2 and 3).

By classical theory, all but finitely many of the  primes $\wp$ of $K$ are of good reduction for 
$\psi$. We denote by ${\cal{P}}_{\psi}$ the collection of these primes, and for each  $\wp \in {\cal{P}}_{\psi}$, we consider the residue field $\F_{\wp}$ at $\wp$ and the
$A$-module structure on $\F_{\wp}$, denoted $\psi(\F_{\wp})$, defined by the reduction $\psi \otimes \wp : A
\longrightarrow \F_{\wp}\{\tau\}$  of $\psi$ modulo $\wp$.

By the theory of torsion points for Drinfeld modules and that of finitely generated modules over a PID,
there exist uniquely determined monic polynomials $d_{1, \wp}(\psi), \ldots, d_{r, \wp}(\psi) \in A$ such that
\begin{equation}\label{elem-div-dm}
\psi(\F_{\wp}) \simeq_{A} A/d_{1, \wp}(\psi) A \times \ldots \times A/d_{r, \wp}(\psi) A
\end{equation}
and
$$
d_{1, \wp}(\psi) | \ldots | d_{r, \wp}(\psi).
$$
The polynomials $d_{1, \wp}(\psi), \ldots, d_{r, \wp}(\psi)$ are  the {\bf{elementary divisors}} of the $A$-module
$\psi(\F_{\wp})$, with the largest of them, $d_r = d_{r, \wp}(\psi)$, the {\bf{exponent}}, having the property that
$d_{r}  x = 0$  for all $x \in  \psi(\F_{\wp})$.
Here, $d_r x :=  (\psi \otimes \F_{\wp})(d_r)(x)$.

Associated to this setting, we introduce the following additional notation.
We let $\F_K$ denote the constant field of $K$ and   $c_K := [\F_K:\F_q]$; thus $\F_K = \F_{q^{c_K}}$. For a non-zero $a \in A$, we let
$|a|_{\infty} := q^{\deg a}$, where $\deg a$ is the degree of $a$ as a polynomial in $T$. For a  prime $\wp$ of $K$, we
let $\deg_K \wp := [\F_{\wp} : \F_K]$ and $|\wp|_{\infty} := q^{c_K \deg_K \wp}$. 
We  set
$$
\pi_K(x)
:=
\#\{\wp \; \text{prime of $K$}: \deg_K \wp = x\}
$$
and   recall the effective Prime Number Theorem for $K$:
\begin{equation}\label{PNT}
 \pi_K(x) = \frac{q^{c_K x}}{x} + \O_K\left( \frac{q^{\frac{c_K x}{2}}}{x} \right).
\end{equation}


The  first  main result of the paper is:
\begin{theorem}\label{thm1}
Let $q$ be a  prime power,  
$A := \F_q[T]$, $k := \F_q(T)$, 
and $K/k$ a finite field extension.
Let $\psi: A \longrightarrow K\{\tau\}$  be a generic Drinfeld $A$-module over $K$, of rank $r \geq 2$. 
Let $d \in A$ be monic.
Then, as $x \to \infty$, 
\begin{equation}\label{formula-thm1}
\#\left\{
\wp \in {\cal{P}}_{\psi}:  \deg_K \wp = x, d_{1,\wp}(\psi) = d 
\right\} 
\sim 
\pi_K(x) 
\ds\sum_{\substack{m \in A \\ m \text{ monic}}} 
\frac{\mu_A(m) c_{md}(x)}{[K(\psi[md]):K]},
\end{equation}
where 
$\mu_A(\cdot)$ is the M\"{o}bius function on $A$, 
 $K(\psi[md])$ is the $md$-division field of $\psi$, 
and
\[
c_{md}(x) 
:=
\left\{
\begin{array}{cc}
[K(\psi[md]) \cap \overline{\F}_K : \F_K] & \text{if \; $[K(\psi[md]) \cap \overline{\F}_K : \F_K] \;  | \;  x$,}
\\
0 & \text{otherwise}.
\end{array}
\right.
\]
Moreover, the Dirichlet density of the set
$$
\left\{
\wp \in {\cal{P}}_{\psi}:  d_{1,\wp}(\psi) = d 
\right\} 
$$
 exists and is given by
\begin{equation} \label{definition-of-d-density}
\delta_{\psi,K}(d)
:=
\ds\sum_{\substack{m \in A \\ m \textrm{ monic}}} \frac{\mu_A(m)}{[K(\psi [md]) : K]}.
\end{equation}
\end{theorem}
\noindent
This is a large generalization of a result  proven independently in \cite{KuLi}.

The essence  of the proof of this theorem  may be summarized as a
Chebotarev Density Theorem  for {\it{infinitely many}} Galois extensions generated by the  generic Drinfeld module $\psi$:
\begin{theorem}\label{thm2}
Let $q$ be a prime power,  $A := \F_q[T]$, $k := \F_q(T)$, and $K/k$ a finite field extension.
Let $\psi: A \longrightarrow K\{\tau\}$  be a generic Drinfeld $A$-module over $K$, of
rank $r \geq 2$. 
Then, as $x \to \infty$, 
$$
\ds\sum_{\substack{m \in A \\ m \text{ monic}}} 
\#\left\{
\wp \in {\cal{P}}_{\psi}:  \deg_K \wp = x, \wp \;  \text{splits completely in } K(\psi[m]) 
\right\} 
\sim 
\pi_K(x) 
\ds\sum_{\substack{m \in A \\ \text{ monic}}} \frac{c_m(x)}{[K(\psi[m]):K]},
$$
with notation as in Theorem \ref{thm1}.  
\end{theorem}

As a consequence of the techniques used in proving Theorems \ref{thm1} and  \ref{thm2},  we obtain the following analogues of  the results of \cite{Du} and \cite{FrKu} in the case of a rank 2 generic Drinfeld module over $K$:
\begin{theorem}\label{thm3}
Let $q$ be a prime power,  $A := \F_q[T]$, $k := \F_q(T)$,
and $K/k$ a finite field extension.
Let $\psi: A \longrightarrow K\{\tau\}$  be a generic Drinfeld $A$-module over $K$, of rank $2$.
\begin{enumerate}
\item[(i)]
Let $f : (0, \infty) \longrightarrow (0, \infty)$ be such that $\ds\lim_{x \rightarrow \infty} f(x) = \infty$.
Then, as  $x \rightarrow \infty$,
$$
\#\left\{
\wp \in {\cal{P}}_{\psi}: \deg_K \wp = x, 
|d_{2, \wp}(\psi)|_{\infty} > \frac{|\wp|_{\infty}}{q^{c_K f(x)}}
\right\}
\sim
\pi_K(x).
$$
Moreover, 
the Dirichlet density of the set
$$
\left\{
\wp \in {\cal{P}}_{\psi}:
|d_{2, \wp}(\psi)|_{\infty} > \frac{|\wp|_{\infty}}{q^{c_K f(\deg_K \wp)}}
\right\}
$$
exists and equals 1.
\item[(ii)]
As $x \rightarrow \infty$, we have the average
$$
\frac{1}{\pi_K(x)}
\ds\sum_{\wp \in {\cal{P}}_{\psi} \atop{\deg_K \wp = x}}
|d_{2, \wp}(\psi)|_{\infty}
\sim
q^{c_K x}
\ds\sum_{m \in A \atop{m \; \text{monic}}}
\frac{c_m(x)}{[K(\psi[m]) : K]}
\ds\sum_{a, b \in A 
\atop{a, b \; \text{monic}
\atop{a b = m}
}
}
\frac{\mu_A(a)}{|b|_{\infty}}.
$$
\end{enumerate}
\end{theorem} 

The structure of the paper is as follows.
In Section 2, we review standard notation and terminology for the arithmetic of $A$ and that of $A$-fields, and we include a few lemmas on elementary function field arithmetic which will be used in Sections 5 and 6.
In Section 3, we discuss definitions and main results from the theory of Drinfeld modules, with a focus on
 properties of division fields of  Drinfeld modules. These properties will be relevant in the proofs of Theorems 1, 2, 3. 
In Section 4, we recall an effective version of the Chebotarev density theorem which we apply to division fields of generic Drinfeld modules, using results from Section 3; this application of Chebotarev is the first key ingredient in the proofs of our main theorems.
In Section 5, we present the proofs of Theorems 1 and 2, while in Section 6 we present the proof of Theorem 3. With the algebraic background from Sections 3-5 in place, the general flavour  of our proofs is analytic. We conclude the paper with remarks on the error terms and the densities appearing in our main theorems.


\section{Notation and basic facts}

Throughout the paper, we will use the following notation and basic results.

\subsection{$\N, \R, \C$ notation.}

We use $\N$ for the set of natural numbers $\{1, 2, 3, \ldots\}$, and $\R$, $\C$ for the sets of real,
respectively complex  numbers.

\subsection{$\O, \ll, \o, \sim$ notation.}

For two functions $f, g: D \longrightarrow \R$, with $D \subseteq \C$ and $g$ positive, we write $f (x)= \O(g(x))$ or $f(x) \ll g(x)$ if there is a positive constant $C$ such that $|f(x)| \leq C g(x)$ for all $x \in D$.
  If $C$ depends on another specified  object $C'$, we write $f(x) = \O_{C'}(g(x))$ or $f(x) \ll_{C'} g(x)$. 
  We write
   $f(x) = \o(g(x))$,
  or sometimes  $f(x) \sim 0 \cdot g(x)$,
    if $\ds\lim_{x \rightarrow \infty} \frac{f(x)}{g(x)} = 0$, and  
  $f(x) \sim \delta  g(x)$ for $\delta > 0$ if $\ds\lim_{x \rightarrow \infty} \frac{f(x)}{g(x)} = \delta$
  (whenever the limits exist).

\subsection{Elementary arithmetic.}



We let $q$ be a prime power, fixed throughout the paper.  Our implied $\O$-constants may depend on $q$, without any additional specification.

We denote by $\F_q$ the finite field with $q$ elements, by
$\F_q^{\ast}$ its group of units, by $\overline{\F}_q$ an algebraic
 closure, and by $\tau: x \mapsto x^q$ the $q$-th power Frobenius automorphism.

As in Section 1, we denote by $A := \F_q[T]$ the polynomial ring over $\F_q$ and  by $k := \F_q(T) = \Quot(A)$ its field of fractions; we  denote by $A^{(1)}$ the set of monic polynomials in $A$.

We recall that $A$ is a Euclidean domain, hence the greatest common divisor, denoted $\gcd$, and the least common
multiple, denoted $\lcm$, exist in $A$. We  recall that $\frac{1}{T}$ plays the role of the ``prime at infinity''
of $k$, while the ``finite primes'' of $k$ are identified with monic irreducible polynomials of $A$. We will simply
refer to the latter as the {\bf{primes of $k$}}.

We denote the monic irreducible elements of $A$ by $p$ or $\ell$.
We denote the primes of $k$ by $\mathfrak{p} = p A$, with $p \in A^{(1)}$,  or by $\mathfrak{l} = \ell A$,
with $\ell \in A^{(1)}$.
For such primes,  we denote 
their residue fields by 
$\F_{\mathfrak{p}}$, $\F_{\mathfrak{l}}$,
and 
the   completions of $A$, respectively of  $k$,
by 
$A_{\mathfrak{p}}$, $A_{\mathfrak{l}}$,
and
$k_{\mathfrak{p}}$,  $k_{\mathfrak{l}}$.

For $a \in A$, we use the standard notation:
\begin{itemize}
\item
$\deg a$ for the degree of $a \neq 0$ as a polynomial in $T$, and  $\deg 0 := - \infty$;
\item
$|a|_{\infty} := q^{\deg a}$ if $a \neq 0$, and $|0|_{\infty} := 0$;
\item
$\sgn(a) \in \F_q$ for the leading coefficient of $a$;
\item
$\mu_A(a)$ for the M\"{o}bius function  of $a$ on $A$; that is, using the notation
 $a = \sgn(a) \cdot p_1^{e_1} \ldots p_t^{e_t}$ 
 for  the prime decomposition of $a \in A \backslash \F_q$, we have
$$
\mu_A(a) :=
\left\{ \begin{array}{cc}
1 & \textrm{if} \; a \in \F_q^{*}, \\
 (-1)^t & \textrm{if  $a \in A \backslash \F_q$ and } e_1 = e_2 = \ldots = e_t = 1,  \\
   0    &             \textrm{otherwise} ;            \\
\end{array}
 \right.
$$
\item
$(A/a A)^{\ast}$ for the group of units of $A/a A$;
\item
$\phi_{A}(a)$ for the Euler function of $a$ on $A$; that is,
\begin{eqnarray*}
\phi_A(a)
&:=&
\#(A/a A)^{\ast}
\\
&=&
\#\{a' \in A \backslash \{0\}: \deg a' < \deg a, \gcd(a, a') = 1\}
\\
&=&
|a|_{\infty}
 \ds\prod_{p \in A^{(1)}  \atop{p | a}}
 \left(1 - \frac{1}{|p|_{\infty}}\right);
\end{eqnarray*}
\item
$\GL_r(A/a A) := \left\{ (a_{i j})_{1 \leq i, j \leq r}: a_{i j} \in A/a A, \det (a_{i j})_{i, j} \in (A/a
A)^{\ast} \right\}. $
\end{itemize}

We record below a few arithmetic results  needed in the proofs of our main theorems.  

\begin{lemma}\label{lemma1}
Let $y \in \N$. Then:
\begin{enumerate}
\item[(i)]
$
\ds\sum_{a \in A^{(1)} \atop{0 \leq \deg a \leq y}
}
1
=
\frac{q^{y+1} - 1}{q-1};
$
\item[(ii)]
$
\ds\sum_{a \in A^{(1)} \atop{0 \leq \deg a \leq y}
}
\deg a
\leq
y \frac{q^{y+1} - 1}{q-1}.
$
\end{enumerate}
\end{lemma}
\begin{proof}
Part (i) is deduced easily by partitioning the polynomials $a$ under summation according to their degrees.
Part (ii) follows from this.
\end{proof}

\begin{lemma}\label{lemma2}
Let $y \in \N \backslash \{1, 2\}$ and let $\alpha > 1$.
Then:
\begin{enumerate}
\item[(i)]
$
\ds\sum_{a \in A \atop{\deg a > y}}
\frac{1}{q^{\alpha \deg a}}
=
\frac{q}{\left(1 - \frac{1}{q^{\alpha - 1}}\right) q^{(\alpha - 1)(y + 1) }};
$
\item[(ii)]
$
\ds\sum_{a \in A \atop{\deg a > y}}
\frac{\log \deg a}{q^{\alpha \deg a}}
\leq
\frac{(q - 1) \log y}{ q (\log q) (\alpha - 1) q^{(\alpha - 1) y}}
+
\frac{q-1}{ q (\log q)^2 (\alpha -1)^2 y} \cdot q^{(1 - \alpha) y},
$

\noindent
provided $y$ is sufficiently large (specifically, $(\alpha - 1) (\log q) y (\log y) > 1$).
\end{enumerate}
\end{lemma}
\begin{proof}
As in the proof of Lemma \ref{lemma1}, part (i) is deduced easily by partitioning the polynomials $a$
under summation according to their degrees.
Part (ii) is deduced proceeding similarly, and also by using integration by parts.


\end{proof}

\begin{lemma}\label{fund-Moeb}
Let $a \in A \backslash \{0\}$. Then
\begin{equation*}
\ds\sum_{d \in A \atop{d \mid a}}
\mu_{A}(d)
=
\left\{ \begin{array}{cc}
q - 1 & \textrm{if} \; a \in \F_q^{*}, \\
   0    &             \textrm{otherwise} .      \\
\end{array}
 \right.
\end{equation*}
\end{lemma}
\begin{proof}
If $a \in \F_q^{\ast}$, then 
$$
\ds\sum_{d \in A \atop{d \mid a}} \mu_A(d) = \ds\sum_{d \in \F_q^{\ast}} 1 = q - 1.
$$
If $a \not\in \F_q^{\ast}$, let 
$$
a = \sgn(a) \cdot p_1^{e_1} \ldots p_t^{e_t}
$$
be the prime factorization of $a$ and let
$$
\text{rad} \; (a) :=  \sgn(a) \cdot p_1 \ldots p_t
$$
be the radical of $a$.
Then 
$$
\ds\sum_{d \in A \atop{d \mid a}}
\mu_A(d)
=
\ds\sum_{d \in A \atop{d \mid \text{rad} \; (a)}}
\mu_A(d).
$$
Note that, if $\F_q^{\ast} = \langle u \rangle$, then the divisors $d \mid \text{rad} \; (a)$ are of the form: 
$u^{\alpha}$ for some $1 \leq \alpha \leq q-1$;
or
$u^{\alpha} p_i$ for some $1 \leq \alpha \leq q-1$ and some $1 \leq i \leq t$;
or
$u^{\alpha} p_{i_1} p_{i_2}$ for some $1 \leq \alpha \leq q-1$ and some $1 \leq i_1 < i_2 \leq t$,
and so on.
For the first, there are $(q-1)  {t  \choose 0} $ possibilities,
for the second there are $(q - 1) {t \choose 1}$ possibilities,
for the third there are $(q - 1) {t \choose 2}$ possibilities,
and so on. In summary, we have
$$
\ds\sum_{d \in A \atop{d \mid \text{rad} \; (a)}}
\mu_A(d)
=
\ds\sum_{0 \leq i \leq t} (q - 1) {t \choose i} (-1)^i
=
(q-1)(1-1)^t
=
0.
$$
This completes the proof.
\end{proof}

\begin{lemma}\label{lemma3}
Let $d \in A \backslash \{0\}$. Then
$$
\frac{1}{|d|_{\infty}}
=
\frac{1}{(q-1)^2}
\ds\sum_{a, b \in A \atop{a b\mid d}}
\frac{\mu_A(a)}{|b|_{\infty}}.
$$
\end{lemma}
\begin{proof}
We have
\begin{eqnarray*}
\frac{1}{(q-1)^2}
\ds\sum_{a, b \in A \atop{a b \mid d}}
\frac{\mu_A(a)}{|b|_{\infty}}
&=&
\frac{1}{(q-1)^2}
\ds\sum_{b \in A}
\frac{1}{|b|_{\infty}}
\ds\sum_{a \in A \atop{a | \frac{d}{b}}}
\mu_A(a)
\\
&=&
\frac{1}{q-1}
\ds\sum_{
b \in A
\atop{\frac{d}{b} \in \F_q^{\ast}}
}
\frac{1}{|b|_{\infty}}
\\
&=&
\frac{1}{|d|_{\infty}},
\end{eqnarray*}
by also using Lemma \ref{fund-Moeb}.
\end{proof}

\begin{lemma}\label{moebius-sum}
Let $d \in A^{(1)}$. Then
$$
\ds\sum_{m, n \in A^{(1)} \atop{m n = d}}
\frac{\mu_A(m)}{|n|_{\infty}}
=
\frac{
(-1)^{\omega(d)} \phi_A(\text{rad} \; (d))
}{
|d|_{\infty}
},
$$
where $\omega(d)$ is the number of all monic prime divisors of $d$ (counted without multiplicities) and 
$\text{rad} \; (d)$ is the radical of $d$.
\end{lemma}
\begin{proof}
By multiplicativity, we have
\begin{eqnarray*}
\ds\sum_{m, n \in A^{(1)} \atop{m n = d}}
\frac{\mu_A(m)}{|n|_{\infty}}
=
\ds\prod_{p \in A^{(1)} \atop{p^t || d}}
\ds\sum_{
n \in A^{(1)}
\atop{n \mid p^t}
}
\frac{ \mu_A \left(  \frac{p^t}{n}  \right)}{|n|_{\infty}}
=
\ds\prod_{p \in A^{(1)} \atop{p^t || d}}
\frac{1 - |p|_{\infty}}{|p|_{\infty}^t}
=
\frac{
(-1)^{\omega(d)} \phi_A(\text{rad} \; (d))
}{
|d|_{\infty}
}.
\end{eqnarray*}
\end{proof}

\subsection{$A$-modules}

For $A$-modules $M_1$, $M_2$, we write $M_1 \simeq_A M_2$ to mean that $M_1$, $M_2$ are isomorphic over $A$, and $M_1 \leq_A M_2$ to mean that $M_1$ is an $A$-submodule of $M_2$.

For a non-zero finite $A$-module $M$, we let $\chi(M)$ be its {\bf Euler-Poincar\'{e} characteristic}, defined as the ideal of $A$ uniquely determined by the conditions:
\begin{enumerate}
\item[(i)] if $M \simeq_A A/\fp$ for a prime ideal $\fp$ of $A$, then $\chi(M) := \fp$;
\item[(ii)] if $0 \to M_1 \to M \to M_2 \to 0$ is an exact sequence of $A$-modules, then $\chi(M) = \chi(M_1) \chi(M_2)$.
\end{enumerate}
We let $|\chi(M)|_{\infty} := |m|_{\infty}$ for some generator $m \in A$ of $\chi(M)$.

\subsection{$A$-fields}

We reserve the notation $(L, \delta)$ for {\bf{$A$-fields}}, that is, pairs consisting of a field $L \supseteq \F_q$
and an $\F_q$-algebra homomorphism $\delta : A \longrightarrow L$. We recall that the kernel of $\delta$ is called the
{\bf{$A$-characteristic of $L$}}; in particular, if $\Ker \delta = (0)$, $L$ is said to have {\bf{generic
$A$-characteristic}}, and if $\Ker \delta \neq (0)$, $L$ is said to have {\bf{finite $A$-characteristic}}.

We denote by $\overline{L}$ an algebraic closure of $L$, by $L^{\text{sep}}$ the separable closure of $L$ in
$\overline{L}$, and by $G_L := \Gal(L^{\text{sep}}/L)$ the absolute Galois group of $L$. We also denote by $L\{\tau\}$
the {\bf{skew-symmetric polynomial ring in $\tau$ over $L$}}, that is,
$$
L\{\tau\} :=\left\{\ds\sum_{0 \leq i \leq n} c_i \tau^i: c_i \in L \; \forall \; 0 \leq i \leq n, n \in \N \cup \{0\}\right\},
$$
with the multiplication rule
$$
\tau c = c^q \tau \; \forall c \in L.
$$
For an element $f \in L\{\tau\}$, we denote by $\deg_{\tau} (f)$ its degree as a polynomial in $\tau$. We recall that
$L\{\tau\}$ is isomorphic to the $\F_q$-endomorphism ring $\End_{\F_q}(\G_a/L)$ of the additive group scheme $\G_a$ over $L$.

\subsection{Finite field extensions of $k$}

We reserve the notation $K$ for a finite field extension of $k$ of genus $g_K$. 
Note that the inclusions $A \subseteq  k$ and $k \subseteq K$ gives $K$ an $A$-field structure of generic $A$-characteristic.
We denote  by $\F_K$ the constant field of $K$ (that is, $\F_K = K \cap
\overline{\F}_q$), and  by $\overline{\F}_K$ an algebraic closure of $\F_K$. We  set  $c_K := [\F_K : \F_q]$.

By a {\bf{prime of $K$}} we mean a discrete valuation ring $\OO$ with maximal ideal $\MM$ such that 
 the quotient field  $\Quot(\OO)$ of $\OO$ equals $K$.  In particular, for a prime $\wp$ of $K$, we denote by $({\cal{O}}_{\wp},
{\cal{M}}_{\wp})$ the associated discrete valuation ring, by $\F_{\wp} := {\cal{O}}_{\wp}/{\cal{M}}_{\wp}$ the
associated residue field, by $\deg_{K} \wp := [\F_{\wp} : \F_K]$ the degree of $\wp$ in $K$, 
and by $\overline{\F}_{\wp}$ an algebraic closure of $\F_{\wp}$.
We denote the prime  $\wp \cap A$ of $A$ by $\mathfrak{p}$, and we denote by $p$ the monic generator of $\mathfrak{p}.$  
 We set  $m_{\wp} := [\F_{\wp} : A/\mathfrak{p} ]$ and record the  diagram:
 \begin{equation} \label{picture-residue-field} 
 \xymatrix{
& \F_{\wp} := \mathcal{O}_{\wp} / \mathcal{M}_{\wp} \ar@{-}[dl]_{\deg_K \wp} \ar@{-}[dr]^{m_{\wp}} & \\
\F_K = \F_{q^{c_K}} \ar@{-}[dr]_{c_K} && \F_{\fp} := A / \mathfrak{p} = \F_{q^{\deg p}} \ar@{-}[dl]^{\deg p} \\
& \F_q & 
}  
\end{equation}
Hence the relationship between the $|\cdot|_{\infty}$-norm of a
prime $\wp$ of $K$ and its associated prime $\fp = pA$ in $k$ is
$$|\wp|_{\infty}= \#\F_K = |p|_{\infty}^{m_{\wp}}.$$

Finally, for a finite Galois extension $K'$ of $K$, we write $\sigma_{\wp}$ for the Artin symbol (``the Frobenius'') at $\wp$ in $K'/K$.

\section{Drinfeld modules}

\subsection{Basic definitions}

Let $(L, \delta)$ be an $A$-field. A {\bf{Drinfeld $A$-module over $L$}} is an $\F_q$-algebra homomorphism
$$
\psi : A \longrightarrow L\{\tau\}
$$
$$
a \longmapsto \psi_a
$$
such that:
\begin{enumerate}
\item[(1)]
for all $a \in A$, $D(\psi_a) = \delta(a)$, where $D : L\{\tau\} \longrightarrow L$,
$D\left(\ds\sum_{0 \leq i \leq n} c_i
\tau^i\right) = c_0$ is the differentiation with respect to $x$ map;
\item[(2)]
$\Im \psi \not\subseteq L$.
\end{enumerate}

A homomorphism $\psi$ as above induces a nontrivial $A$-module structure on $L$, or, more generally,  on
any $L$-algebra $\Omega$; we denote this structure by $\psi(\Omega)$.

Associated to a Drinfeld $A$-module $\psi$ over $L$ we have two important invariants, called the rank and the height.
We define the  {\bf{rank}} of $\psi$ as the unique positive integer $r$ such that
$$
\deg_{\tau}(\psi_a) = r \deg a \; \;  \forall a \in A.
$$
If $L$ has generic $A$-characteristic, we define the {\bf{height}} of $\psi$ as zero. If $L$ is of
finite $A$-characteristic $\mathfrak{p} = pA$, we define the {\bf{height}} of $\psi$ as the unique positive integer $h$ such
that
$$
\min\{0 \leq i \leq r \deg a: c_{i, a}(\psi) \neq 0\} = h \ord_{\mathfrak{p}}(a) \deg p
\; \; \forall a \in A, a \neq 0,
$$
where
$$\psi_a = \ds\sum_{0 \leq i \leq r \deg a} c_{i, a}(\psi) \tau^i$$
 and $\ord_{\fp}(a) := t$ with $p^t || a$.

For the purpose of this paper, the rank and the height are particularly relevant in determining the structure
(\ref{elem-div-dm}) of the reductions modulo primes of a Drinfeld $A$-module in generic characteristic.

\subsection{Endomorphism rings}

Let $(L, \delta)$ be an $A$-field.
Given  $\psi, \psi' : A \longrightarrow L\{\tau\}$ Drinfeld $A$-modules over $L$, a
{\bf{morphism from $\psi$ to $\psi'$ over $L$}} is an element $f \in L\{\tau\}$ such that
$$
f \psi_a = \psi'_a f \; \; \forall a \in A.
$$
An {\bf{isogeny from $\psi$ to $\psi'$ over $L$}} is a non-zero morphism as above. 
An {\bf{isomorphism from $\psi$ to $\psi'$ over $L$}} is an element $f \in L^{\ast}$ such that $f \psi_a = \psi'_a f \;
\; \forall a \in A$. 
Finally, $\End_{L}(\psi)$ and
$\End_{\overline{L}}(\psi)$ are the rings of endomorphisms of $\psi$ over $L$ and over $\overline{L}$, respectively.

We remark that
$$
\psi(A) \subseteq \End_{L}(\psi) \subseteq \End_{\overline{L}}(\psi)
$$
and that isogenous Drinfeld modules have the same rank and height.

For ease of notation, we shall henceforth denote the category of Drinfeld $A$-modules over $L$ 
(with a fixed $A$-field structure $(L, \delta)$)
by
$$\Drin_{A}(L).$$

Note that, in the setting of our main theorems,  $K \supseteq k$ a fixed field extension and $\psi: A \longrightarrow K\{\tau\}$ a generic Drinfeld $A$-module over $K$, we are implicitly working with  the structure on $K$ arising from the injective homomorphism $D \circ \psi : A \longrightarrow K$. 

\begin{remark}\label{general-dm}
{\emph{
The category $\Drin_A(L)$ of Drinfeld $A$-modules over $L$ may be defined in greater generality.  Indeed, we may fix an arbitrary function field ${\mathscr{K}}$ and a prime $\infty$ of ${\mathscr{K}}$.  We then take ${\mathscr{A}}$ as the ring of functions on ${\mathscr{K}}$ regular away from $\infty$ and define the category $\Drin_{\mathscr{A}}(\mathscr{L})$ of Drinfeld ${\mathscr{A}}$-modules over ${\mathscr{A}}$-fields ${\mathscr{L}}$ exactly as we did above.
}}
\end{remark}

The endomorphism rings introduced here have important properties, such as:
\begin{theorem}( \cite[Prop. 4.7.6 p. 80, Theorem 4.7.8 p. 81, Prop. 4.7. 17 p. 84]{Go}, \cite[Theorem 2.7.2, p. 50]{Th}) 
\label{end-ring-theorem2}

\noindent
Let $(L,\delta)$ be an $A$-field with generic $A$-characteristic.  Let $\psi \in \Drin_A(L)$ be of rank $r \geq 1$.  Then:
\begin{enumerate}
\item[(i)]
$\End_{\overline{L}}(\psi)$ is commutative;
\item[(ii)]
$\End_{\overline{L}}(\psi)$ is a finitely generated projective $A$-module of rank at most $r$;
\item[(iii)]
if we denote by \[ k' := \End_{\overline{L}}(\psi) \otimes_A k,\]  then $k'$ is a finite field extension of $k$ satisfying $[k':k] \leq r$.
\end{enumerate}
\end{theorem}


\subsection{Division points}

Let $(L, \delta)$ be an $A$-field and let $\psi \in \Drin_{A}(L)$. Let $a \in A \backslash \F_q$. We define the
{\bf{$a$-division module of $\psi$}} by
$$
\psi[a] :=\left\{\lambda \in {\overline{L}}: \psi_a(\lambda) = 0\right\}.
$$
When $a = \ell$ is irreducible, we define the {\bf{$\ell^{\infty}$-division module of $\psi$}} by
$$
\psi\left[\ell^{\infty}\right] 
:= \ds\bigcup_{n \geq 1} \psi\left[\ell^n\right].
$$

Note that $\psi[a]$ is a torsion $A$-module via $\psi$. As we recall below, its $A$-module structure is well
determined by $a$ and $\psi$.
\begin{theorem} \label{structure-theorem-torsion-points} \cite[Theorem 13.1 p. 221]{Ro}

\noindent
Let $(L, \delta)$ be an $A$-field with $A$-characteristic $\mathfrak{p}$ (possibly zero). Let $\psi \in \Drin_{A}(L)$
be of rank $r \geq 1$ and height $h$.
 Let $\fl \neq \fp$ be a non-zero prime ideal of $A$ with $\fl = \ell A$,  and let  $e \geq 1$ be an integer. Then
$$
\psi[\ell^e] \simeq_{A} (A/\ell^e A)^{r}.
$$
If $\fp = pA$ is non-zero, then \[\psi[p^e] \simeq_A (A / p^eA)^{r-h}.\]
\end{theorem}


\begin{corollary}\label{div-module}

\noindent
Let $(L, \delta)$ be an $A$-field with $A$-characteristic $\mathfrak{p}$ (possibly zero).
Let $\psi \in \Drin_{A}(L)$ be of rank $r \geq 1$ and height $h$.
Let $a \in A \backslash \F_q$ and write the ideal $aA$ (uniquely) as the product of ideals
$\mathfrak{a}_1,\mathfrak{a}_2$ of $A$ such that $\mathfrak{a}_1$ is relatively prime to $\fp$ and $\mathfrak{a}_2$ is composed of prime divisors of $\fp$.
Then
\begin{equation*}
\psi[a]
\simeq_A
(A/\mathfrak{a}_1)^r \oplus (A / \mathfrak{a}_2)^{r-h}
\leq_A (A / a A)^r.
\end{equation*}
\end{corollary}

\begin{remark} \label{remark-after-structure-theorem}
{\emph{
Theorem \ref{structure-theorem-torsion-points} and  Corollary \ref{div-module} hold in greater generality.  In particular, the results hold for a Drinfeld module
$\psi \in \Drin_{\mathscr{A}}({\mathscr{L}})$, where ${\mathscr{A}}$ is the ring of functions on an arbitrary function field ${\mathscr{K}}$ which are regular away from some fixed prime in ${\mathscr{K}}$, and where
${\mathscr{L}}$ is any ${\mathscr{A}}$-field.
}}
\end{remark}

\subsection{Reductions modulo primes}\label{subsection-reductions-modulo-primes}

Let $(L, \delta)$ be an $A$-field and let $\psi \in \Drin_{A}(L)$ be of rank $r \geq 1$.
Let $\wp$ be a prime of $L$.

We say that {\bf{$\psi$ has integral coefficients at $\wp$}} if:
\begin{enumerate}
\item[(1)]
$\psi_a \in {\cal{O}}_{\wp}\{\tau\} \; \; \forall a \in A$;
\item[(2)]
$\psi \otimes \F_{\wp} : A \longrightarrow \F_{\wp}\{\tau\}$, defined by $a \mapsto \psi_a (\mod \wp)$, is a
Drinfeld $A$-module over $\F_{\wp}$ (of some  rank $0 < r_1 \leq r$).
\end{enumerate}
In this case, we also say that $\psi \otimes \F_{\wp}$ is the {\bf{reduction of $\psi$ modulo $\wp$}}.

We say that {\bf{$\psi$ has good reduction at $\wp$}} if there exists $\psi' \in \Drin_{A}(L)$ such that:
\begin{enumerate}
\item[(1)]
$\psi' \simeq_{K} \psi$;
\item[(2)]
$\psi'$ has integral coefficients at $\wp$;
\item[(3)]
$\psi' \otimes \F_{\wp}$ has rank $r$.
\end{enumerate}

For the remainder of this subsection, we assume that $L = K$ is a finite field extension of $k$
and that $\psi \in \Drin_{A}(K)$ has generic characteristic.
There are only finitely many primes of $K$ which are {\it{not}} of good
reduction for $\psi \in \Drin_{A}(K)$. 
As in Section 1, we let ${\cal{P}}_{\psi}$ denote the set of (finite) primes of $K$ of good reduction for 
$\psi$.

Note that, for a prime $\wp \in {\cal{P}}_{\psi}$, Corollary \ref{div-module} gives the structure (\ref{elem-div-dm}) of the $A$-module $\psi(\F_{\wp})$.  Indeed, $\psi(\F_{\wp})$ is a finite $A$-module, and since $A$ is a PID, there exist unique polynomials $d_{1,\wp}(\psi),d_{2,\wp}(\psi),\ldots,d_{s,\wp}(\psi) \in A^{(1)}$ such that 
\[
\psi(\F_{\wp}) 
\simeq_A 
A / d_{1,\wp}(\psi) A \times \ldots \times A / d_{s,\wp}(\psi) A,
\] 
with $d_{i,\wp}(\psi)|d_{i+1,\wp}(\psi)$ for all $i = 1,\ldots,s-1$.  
That $s = r$ follows from the fact that $\psi(\F_{\wp})$ is a torsion module, 
and hence, by Corollary \ref{div-module},
from  the existence of  some  $a \in A \backslash \F_q$ such that
$\psi(\F_{\wp}) \leq _A \psi[a] \leq_A (A / aA)^r$.

The following analogue of the criterion of N\'{e}ron-Ogg-Shafarevich for elliptic curves holds:

\begin{theorem} \label{takahashi} \cite[Theorem 1, p. 477]{Tak}

\noindent
Let $K$ be a finite extension of $k$. 
Let $\psi \in \Drin_A(K)$ be of generic characteristic. 
Let $\wp$ be a prime of $K$ and let  $\fl = \ell A$ be  a prime ideal different from $\fp := \wp \cap A$.  Then $\psi$  has good reduction at $\wp$ if and only if the Galois module $\displaystyle \psi[\ell^{\infty}]$ is unramified at $\wp$. Moreover,  if $\psi$ has rank 1, then $\displaystyle \psi[\ell^{\infty}] $ is totally ramified at $\fl$.
  \end{theorem}
Note that while the last assertion of the theorem is not stated explicitly in \cite[Theorem 1, p. 477]{Tak}, it can be derived from the proof.

\begin{remark} \label{remark-after-takahashi} 
\emph{
The notion of good reduction can be introduced  for a  general $\psi \in \Drin_{\mathscr{A}}({\mathscr{L}})$, where ${\mathscr{A}}$ is the ring of functions on an arbitrary function field ${\mathscr{K}}$ which are regular away from some fixed prime in ${\mathscr{K}}$, and ${\mathscr{L}}$ is a generic ${\mathscr{A}}$-field. Theorem \ref{takahashi} holds in this general setting also.
} 
\end{remark}

\subsection{Division fields}\label{subsection-division-fields}

Let $K$ be a finite field extension of $k$ and let  $\psi \in \Drin_{A}(K)$ be of generic characteristic.  
For  $a \in A$, we define
the {\bf{$a$-division field of
$\psi$}} as  $K(\psi[a])$.
This is a  Galois extension of $K$ which plays a crucial role in our study of the elementary
divisors of the reductions of $\psi$.  We denote the genus of $K(\psi[a])$ by $g_a$ and the  degree of the constant field of $K(\psi[a])$ over $\F_K$ by $c_a$, that is,
\begin{equation} \label{c-a}
c_a := [K(\psi[a]) \cap \overline{\F}_K : \F_K].
\end{equation}
Below are important properties of these division fields.

\begin{proposition} \label{constant-field-size-proposition} \cite[Remark 7.1.9, p. 196]{Go}

\noindent
Let $K$ be a finite field extension of $k$ and let  $\psi \in \Drin_{A}(K)$ be of generic characteristic.  
Let
$$ K_{\psi, \text{tors}} := \ds\bigcup_{a \in A \backslash \F_q} K(\psi[a]). $$
Then
$$[K_{\psi, \text{tors}} \cap \overline{\F}_K : \F_K] < \infty.$$
In particular, there exists a constant $C(\psi,K) \in \N$, depending on $K$ and $\psi$,  such that,
for any $a \in A \backslash \F_q$,
$$c_a \leq C(\psi,K).$$
\end{proposition}

\begin{proposition} \label{genus-proposition}\cite[Corollary 7, p. 248]{Ga}

\noindent
Let $K$ be a finite field extension of $k$ and let  $\psi \in \Drin_{A}(K)$ be of generic characteristic.
Then  there exists a constant $G(\psi,K) \in \N$, depending  on $K$ and $\psi$,  such that,
for any $a \in A \backslash \F_q$,
$$
g_{a} \leq G(\psi, K) \cdot [K(\psi[a]) : K] \cdot \deg a.
$$
\end{proposition}

\subsection{Galois representations} \label{subsection-galois-representations} 

We start with a more general setting, as follows.
Let $\mathscr{K}$ be a finitely generated field of transcendence degree 1 over $\F_q$,
let $\infty$ be a fixed prime of $\mathscr{K}$, and
let $\mathscr{A}$ be the ring of functions on $\mathscr{K}$ regular away from $\infty$.
Let $\mathscr{L}$ be a finitely generated extension of $\mathscr{K}$.
Let $\psi \in \Drin_{\mathscr{A}}(\mathscr{L})$  be of rank $r \geq 1$, automatically of generic characteristic.

Using the general notions of division points on Drinfeld modules, for any non-zero prime 
$\mathfrak{l}$ of $\mathscr{A}$ we define the {\bf{$\mathfrak{l}$-adic Tate module of $\psi$}} by
$$
T_{\mathfrak{l}}(\psi) := 
\Hom_{A}\left(
\mathscr{L}_{\mathfrak{l}}/\mathscr{A}_{\mathfrak{l}}, 
\psi\left[\mathfrak{l}^{\infty}\right]
\right),
$$
where:
 for a non-zero ideal $\mathfrak{a}$ of $\mathscr{A}$,
$\psi[\mathfrak{a}] 
:= 
\left\{
\lambda \in \mathscr{L}^{\text{sep}} :
\psi_{a} (\lambda) = 0 \; \forall a \in \mathfrak{a}
\right\}$;
$\psi\left[\mathfrak{l}^{\infty}\right] :=  \ds\cup_{n \geq 1} \psi\left[\mathfrak{l}^n\right]$;
and 
$\mathscr{L}_{\mathfrak{l}}$, $\mathscr{A}_{\mathfrak{l}}$ are the respective $\mathfrak{l}$-completions.

The $\mathfrak{l}$-adic Tate module of $\psi$ is a free $\mathscr{A}_{\mathfrak{l}}$-module of rank $r$.
Moreover, it gives rise to continuous Galois representations
$$
\rho_{\mathfrak{l}, \psi} :
G_{\mathscr{L}}
\longrightarrow 
\Aut_{\mathscr{A}_{\mathfrak{l}}} (T_{\mathfrak{l}}(\psi)) \simeq \GL_r(\mathscr{A}_{\mathfrak{l}}),
$$
$$
\rho_{\psi} : 
G_{\mathscr{L}} 
\longrightarrow 
\ds\prod_{\mathfrak{l} \neq \infty}
\Aut_{\mathscr{A}_{\mathfrak{l}}} (T_{\mathfrak{l}}(\psi)) 
\simeq 
\ds\prod_{\mathfrak{l} \neq \infty}
\GL_r(\mathscr{A}_{\mathfrak{l}})
\simeq
\GL_r\left(\hat{\mathscr{A}}\right)
$$
of the absolute Galois group $G_{\mathscr{L}} := \Gal(\mathscr{L}^{\text{sep}}/\mathscr{L})$
of $\mathscr{L}$. 
Here, $\hat{\mathscr{A}} := \ds\lim_{\leftarrow \atop{\mathfrak{a}}} \mathscr{A}/\mathfrak{a}$, where 
$\mathfrak{a}$ 
are non-zero ideals of $\mathscr{A}$  ordered by divisibility.

These representations fit into a commutative diagram
$$
 \begin{array}{ccc}
\xymatrix{
G_{\mathscr{L}}  \ar[r]^(.35){\rho_{\psi}} \ar[dr]^{\rho_{{\mathfrak{l}}, \psi}} \ar[rdd]_{\bar{\rho}_{{\mathfrak{l}^n}, \psi}}
& \ds\prod_{\mathfrak{l} \neq \infty} \GL_r(\mathscr{A}_{\mathfrak{l}})  \ar[d]^{\pi}
\\
       & \GL_r(\mathscr{A}_{{\mathfrak{l}}}) \ar[d]^{\mod {\mathfrak{l}^n}}
\\
       & \GL_r(\mathscr{A}/{\mathfrak{l}^n} \mathscr{A}),
}
\end{array}
$$  
with $\pi$ denoting the natural projection
and $\mod {\mathfrak{l}^n}$ denoting the reduction modulo ${\mathfrak{l}^n}$ map.

Since the {\it{residual}} representation $\bar{\rho}_{\mathfrak{l}^n, \psi}$ gives rise to an 
{\it{injective}} representation
$$
\bar{\rho}_{\mathfrak{l}^n, \psi} : 
\Gal(\mathscr{L}(\psi\left[\mathfrak{l}^n\right])/\mathscr{L}) 
\hookrightarrow 
\GL_r(\mathscr{A}/\mathfrak{l}^n \mathscr{A}),
$$
we immediately deduce the upper bound
\begin{equation}\label{division-upper-1}
[\mathscr{L}(\psi\left[\mathfrak{l}^n\right]) : \mathscr{L}] 
\leq 
\#\GL_r(\mathscr{A}/\mathfrak{l}^n \mathscr{A}).
\end{equation}
This bound may be better understood using:

\begin{lemma}  \label{lemma-breuer}

\noindent
Let $\mathscr{A}$ be a Dedekind domain whose field of fractions is a global field $\mathscr{K}$.
Let $\mathfrak{a}$ be a non-zero ideal of $\mathscr{A}$. Define 
$$|\mathfrak{a}| := \#(\mathscr{A}/\mathfrak{a}).$$  
Then, for any $r \in \N$, we have 
\begin{eqnarray*}
\#\GL_r(\mathscr{A}/\mathfrak{a}) 
&=& 
|\mathfrak{a}|^{r^2} 
\ds\prod_{\substack{\mathfrak{l}|\mathfrak{a} \\ \mathfrak{l} \textrm{ prime}}} 
\left(1 - \frac{1}{|\mathfrak{l}|} \right)
 \left(1 - \frac{1}{|\mathfrak{l}|^2} \right) 
 \ldots
  \left( 1 - \frac{1}{|\mathfrak{l}|^r} \right)
\\
&\gg_{\mathscr{K}}&
 \frac{|\mathfrak{a}|^{r^2}}{\log \log |\mathfrak{a}|}.
\end{eqnarray*}
\end{lemma}
\begin{proof}
This is deduced   from \cite[Lemma 2.2, p. 1243]{Br} and \cite[Lemma 2.3, p. 1244]{Br}.
\end{proof}

For the main results of this paper we require more precise information about the degree
$[\mathscr{L}\left(\psi\left[\mathfrak{l}^n\right]\right) : \mathscr{L}] $, which we deduce from the important results of R. Pink and E. R\"{u}tsche  \cite{PiRu}.  More precisely, we first recall:
 
\begin{theorem} \cite[Theorem 0.1, p. 883]{PiRu} \label{pink-rutsche1}

\noindent
We keep the setting introduced at  the beginning of Section 3.6 and assume that
$$
\End_{\overline{\mathscr{L}}}(\psi) = \mathscr{A}.
$$
Then the image of the representation $\rho_{\psi}$ is open in $\GL_r\left(\hat{\mathscr{A}}\right)$, that is,
$$
|\GL_r\left(\hat{\mathscr{A}}\right) : \Im \rho_{\psi}| < \infty.
$$
In particular,
there exists an integer $i_1(\psi, \mathscr{L}) \in \N$ such that, for any non-zero $a \in \mathscr{A}$,
 \begin{equation*}\label{division-lower1}
 \left| 
 \GL_r(\mathscr{A}/a\mathscr{A}) : \Gal(\mathscr{L}(\psi[a]) / \mathscr{L}) 
 \right| 
 \leq 
 i_1(\psi, \mathscr{L}),
 \end{equation*}
 and there exists an ideal $I_1(\psi, \mathscr{L})$ of $\mathscr{A}$ such that
 for any non-zero $a \in \mathscr{A}$ with $(a \mathscr{A}, I_1(\psi, \mathscr{L})) = 1$,
 \begin{equation*}
 \Gal(\mathscr{L}(\psi[a]) / \mathscr{L})
 \simeq
 \GL_r(\mathscr{A}/a\mathscr{A}).
 \end{equation*}
 \end{theorem}

Note that $\psi$ may have a non-trivial endomorphism ring.
If all endomorphisms of $\psi$ are actually defined over $\mathscr{L}$, then
the image of $\rho_{\mathfrak{l}, \psi}$ lies in the centralizer 
$\Centr_{ \GL_r(\mathscr{A}_{\mathfrak{l}})}\left(\End_{\overline{\mathscr{L}}}(\psi)\right)$. 
In this case, we focus on  the representations
$$
\rho_{\mathfrak{l}, \psi} :
G_{\mathscr{L}}
\longrightarrow 
\Centr_{ \GL_r(\mathscr{A}_{\mathfrak{l}})}\left(\End_{\overline{\mathscr{L}}}(\psi)\right),
$$
$$
\rho_{\psi} : 
G_{\mathscr{L}} 
\longrightarrow 
\ds\prod_{\mathfrak{l} \neq \infty}
\Centr_{ \GL_r(\mathscr{A}_{\mathfrak{l}})}\left(\End_{\overline{\mathscr{L}}}(\psi)\right),
$$
and recall:

\begin{theorem}\cite[Theorem 0.2, p. 883]{PiRu}\label{pink-rutsche2}

\noindent
We keep the  setting introduced at  the beginning of Section 3.6 and assume that
$$
\End_{\overline{\mathscr{L}}}(\psi) = \End_{\mathscr{L}}(\psi).
$$
Then the image of the  representation $\rho_{\psi}$ is open in 
$\ds\prod_{\mathfrak{l} \neq \infty}
\Centr_{ \GL_r(\mathscr{A}_{\mathfrak{l}})}\left(\End_{\overline{\mathscr{L}}}(\psi)\right)$,
that is,
$$
\left|
\ds\prod_{\mathfrak{l} \neq \infty}
\Centr_{ \GL_r(\mathscr{A}_{\mathfrak{l}})}\left(\End_{\overline{\mathscr{L}}}(\psi)\right)
: \Im \rho_{\psi}
\right| 
< \infty.
$$
In particular,
there exists an integer $i_2(\psi, \mathscr{L}) \in \N$ such that, for any non-zero $a \in \mathscr{A}$,
 \begin{equation*}\label{division-lower2}
 \left| 
\Centr_{ \GL_r(\mathscr{A}/a \mathscr{A})}\left(\End_{\overline{\mathscr{L}}}(\psi)\right)
 : 
 \Gal(\mathscr{L}(\psi[a]) / \mathscr{L}) 
 \right| 
 \leq 
 i_2(\psi, \mathscr{L}),
\end{equation*}
 and there exists an ideal $I_2(\psi, \mathscr{L})$ of $\mathscr{A}$ such that
 for any non-zero $a \in \mathscr{A}$ with $(a \mathscr{A}, I_2(\psi, \mathscr{L})) = 1$,
 \begin{equation*}
 \Gal(\mathscr{L}(\psi[a]) / \mathscr{L})
 \simeq
 \Centr_{ \GL_r(\mathscr{A}/a \mathscr{A})}\left(\End_{\overline{\mathscr{L}}}(\psi)\right).
 \end{equation*}
\end{theorem}

We will apply these results to deduce a lower bound for the degree $[K(\psi[a]) : K]$ of the $a$-division field of a generic Drinfeld module $\psi \in \Drin_{A}(K)$ of rank $r \geq 2$, where $K$ is a finite extension of $k$. Before we state and prove this bound, let us recall the Drinfeld module analogue of the Tate Conjecture,
proven independently in \cite{Tag1}-\cite{Tag2} and  \cite{Tam}:

\begin{theorem}(The Tate Conjecture for Drinfeld modules)\label{tate-conjecture}

\noindent
We keep the previous general setting $\mathscr{K}$, $\mathscr{A}$, $\mathscr{L}$.
Let $\psi_1$, $\psi_2 \in \Drin_{\mathscr{A}}(\mathscr{L})$.
Then, for any prime $\mathfrak{l}$ of $\mathscr{A}$, the natural map
$$
\Hom_{\mathscr{L}}(\psi_1, \psi_2) \otimes_{\mathscr{A}}(\mathscr{A}_{\mathfrak{l}})
\longrightarrow
\Hom_{\mathscr{A}_{\mathfrak{l}}[G_{\mathscr{L}}]}
\left(T_{\mathfrak{l}}(\psi_1), T_{\mathfrak{l}}(\psi_2)\right)
$$ 
is an isomorphism.
\end{theorem} 

We are now ready to prove:
\begin{theorem}\label{division-degree}
Let $K$ be a finite extension of $k$ and let $\psi \in \Drin_{A}(K)$ be of rank $r \geq 2$ and of 
 generic characteristic.
Let $\gamma := \text{rank}_{A}  \; \End_{\overline{K}}(\psi)$.
Then, for any $a \in A \backslash \F_q$, we have 
$$
\frac{
|a|_{\infty}^{   \frac{r^2}{\gamma}   }
}
{\log \gamma + \log \deg a + \log \log q}
\ll_{\psi, K}
[K(\psi[a]) : K]
\leq
|a|_{\infty}^{\frac{r^2}{\gamma}}.
$$
\end{theorem}
\begin{proof}
We base the proof on  a  strategy used in \cite{Pi}, as follows.
Let $\tilde{A} :=  \End_{\overline{K}}(\psi)$ and  let $F$  be the field of fractions of $\tilde{A}$.  
By Theorem \ref{end-ring-theorem2}, all endomorphisms $f \in \tilde{A}$ are defined over a finite extension 
$\tilde{K}$ of $K$.  
Thus, after identifying $A$ with its image $\psi(A) \subseteq \tilde{A}$, we can extend
 $\psi: A \longrightarrow K\{\tau\}$
 tautologically to a homomorphism
$$
\tilde{\psi} : \tilde{A} \longrightarrow \tilde{K}\{\tau\}.
$$
This is again a Drinfeld module, with the difference that $\tilde{A}$ may not be a {\it{maximal}} order in $\tilde{K}$.
To fix this, we modify $\tilde{\psi}$ by a suitable isogeny, using results of  D. Hayes \cite{Ha}.

Indeed, we let
$\mathscr{A}$ be  the normalization of $\tilde{A}$ in $\tilde{K}$.
Then, by \cite[Proposition 3.2, p. 182]{Ha},
there exists a Drinfeld module
$$
\overline{\psi}:  \mathscr{A} \longrightarrow \overline{K}\{\tau\}
$$
such that $\overline{\psi}|_{\tilde{A}}$ is $K$-isogenous to $\tilde{\psi}$. 
Moreover, $\overline{\psi}$ may be chosen such that the restriction $\overline{\psi}|_{\tilde{A}}$ is defined over $K$.

Let $\mathscr{K}$ be the finite field extension of $K$ generated by the coefficients of all endomorphisms in
$\End_{\overline{K}}(\overline{\psi})$. 
By Theorem \ref{tate-conjecture},  all the endomorphisms of $\overline{\psi}$ over $\overline{K}$ are defined already over $K^{\text{sep}}$.
Thus $\mathscr{K}$ is a separable Galois extension of $K$.
Moreover, by construction, the Galois group $\Gal(\mathscr{K}/K)$ acts on $F$ and, again by Theorem \ref{tate-conjecture}, it acts faithfully.

Let
$$
\Psi :  \mathscr{A} \longrightarrow \mathscr{K}\{\tau\}
$$
be the tautological extension of $\overline{\psi}$.
This is a generic Drinfeld module of rank
$$
R :=  \frac{r}{\gamma}
$$
satisfying $\End_{\overline{\mathscr{K}}}(\Psi) = \End_{\mathscr{K}}(\Psi) = \mathscr{A}$.
By Theorem \ref{pink-rutsche2}, the image of the representation
$$
\rho_{\Psi}:
G_{\mathscr{K}}
\longrightarrow
\ds\prod_{\mathfrak{l} \neq \infty}
\Centr_{ \GL_{R}({\mathscr{A}}_{\mathfrak{l}}) }(\mathscr{A})
\simeq
\ds\prod_{\mathfrak{l} \neq \infty}
\GL_{R}({\mathscr{A}}_{\mathfrak{l} })
$$
is open.
In particular,  
there exists an integer $i(\Psi, \mathscr{K}) \in \N$ such that, for any non-zero $a \in \mathscr{A}$,
 \begin{equation}\label{division}
 \left| 
 \GL_R(\mathscr{A}/a\mathscr{A}) : \Gal(\mathscr{K}(\Psi[a]) / \mathscr{K}) 
 \right| 
 \leq 
 i(\Psi, \mathscr{K}).
 \end{equation}

Since $\mathscr{A}$ is a Dedekind domain, by Lemma \ref{lemma-breuer} we deduce that, for any 
non-zero $a \in \mathscr{A}$,
\begin{equation}\label{div2}
\frac{|a \mathscr{A}|^{R^2}}{\log \log |a \mathscr{A}|}
\ll_{\mathscr{K}}
\# \GL_{R}(\mathscr{A}/a \mathscr{A})
\leq
|a \mathscr{A}|^{R^2},
\end{equation}
where $|a \mathscr{A}| := \#(\mathscr{A}/a \mathscr{A})$.

Now let $a \in A \backslash \F_q$ and 
remark that $\#(\mathscr{A}/a \mathscr{A}) = \#(A/a A)^{\gamma} = |a|_{\infty}^{\gamma}$. Therefore 
(\ref{division}) and (\ref{div2}) imply that
\begin{equation*}
\frac{
|a|_{\infty}^{\frac{r^2}{\gamma}}
}{
\log \gamma + \log \deg a + \log \log q
}
\ll_{\Psi, \mathscr{K}}
[\mathscr{K}(\Psi[a]) : \mathscr{K}]
\leq
|a|_{\infty}^{\frac{r^2}{\gamma}}.
\end{equation*}
Finally, recalling the construction and properties of $\tilde{\psi}, \overline{\psi}$ and $\Psi$ in relation to $\psi$, these bounds imply  the ones stated in the theorem. 
\end{proof}

\subsection{Arithmetic in division fields}

Let $K$ be a finite field extension of $k$ and let  $\psi \in \Drin_{A}(K)$ be  of generic characteristic and rank $r \geq 1$.   
We  focus on providing useful characterizations and properties of the primes splitting completely in  the division fields of $\psi$. 

Let $\wp \in {\cal{P}}_{\psi}$ and let $\mathfrak{l} = \ell A$ be a prime of $k$ such that 
$\mathfrak{l} \neq \wp \cap A$. 
Let $\sigma_{\wp}$ denote the Frobenius at  $\wp$ in $K(\psi[\ell]) / K$.
The
{\bf{characteristic poynomial of the Frobenius $\sigma_{\wp}$ at $\wp$}}, 
defined by
\begin{eqnarray*}
P_{\psi, \wp}^{\mathfrak{l}}(X)
&:=&
\det (X \ \text{Id} - \rho_{\psi, \mathfrak{l}}(\sigma_{\wp}))
\\
&=&
X^r + a_{r-1, \wp}(\psi) X^{r-1} + \ldots + a_{1, \wp}(\psi) X + a_{0, \wp}(\psi)
\in A_{\mathfrak{l}}[X],
\end{eqnarray*}
is very useful in describing further properties of  $\wp$ when it splits completely in a division 
field of $K$.
We recall the basic  properties of this polynomial:
\begin{theorem}\cite[Corollary 3.4, p. 193; Theorem 5.1, p. 199]{Ge}
\label{gekeler1}

\noindent
Let $K$ be a finite field extension of $k$ and let  $\psi \in \Drin_{A}(K)$ be of generic characteristic and rank $r \geq 1$.  
Let $\wp \in {\cal{P}}_{\psi}$ and let $\mathfrak{l} = \ell A$ be a prime of $k$ such that 
$\mathfrak{l} \neq \wp \cap A$. 
Then:
\begin{enumerate}
\item[(i)]
$P_{\psi, \wp}^{\fl}(X)  \in A[X]$;
in particular, $P_{\psi, \wp}^{\fl}(X)$ is independent of $\fl$, and, as such, we may drop the superscript $\fl$ from notation and simply write $P_{\psi, \wp}(X)$. 
\item[(ii)]
There exists $u_{\wp}(\psi) \in \F_q^*$ such that 
$a_{0, \wp}(\psi) = u_{\wp}(\psi) p^{m_{\wp}}$, where, we recall, $m_{\wp} := [\F_{\wp} : \F_{\fp}]$.
\item[(iii)]
The roots of $P_{\psi, \wp}(X)$ have $|\cdot|_{\infty}$-norm less than or equal to
$|\wp|_{\infty}^{\frac{1}{r}}$.
\item[(iv)]
$|a_{i, \wp}(\psi)|_{\infty} \leq |\wp|_{\infty}^{\frac{r-i}{r}}$ for all $0 \leq i \leq r-1$.
\item[(v)]
$P_{\psi, \wp}(1) A = \chi(\psi(\F_{\wp}))$, where, we recall, $\chi(\psi(\F_{\wp}))$ denotes the 
Euler-Poincar\'{e} characteristic of $\psi(\F_{\wp})$.
\end{enumerate}
\end{theorem}


\begin{proposition} \label{sc-iff}(Characterization of primes splitting completely in division fields)

\noindent
Let $K$ be a finite field extension of $k$ and let $\psi \in \Drin_{A}(K)$ be of generic characteristic and 
 rank $r \geq 1$.
Let $\wp \in {\cal{P}}_{\psi}$ and  let $m \in A^{(1)}$ be such that $\gcd(m,p) = 1$, where $\wp \cap A = p A$.
Then
 $\psi(\F_{\wp})$ contains an $A$-submodule
 isomorphic to $(A/mA)^r$ if and only if $\wp$ splits completely in $K(\psi[m]) / K$.  Consequently,  given $d \in A^{(1)}$ with $\gcd(d, p) = 1$, we have that
 $d_{1,\wp}(\psi) = d$ if and only if
 $\wp$ splits completely in $K(\psi[d])/K$ and 
 $\wp$ does not split completely in $K(\psi[d \ell]) / K$ for any prime $\ell \in A$ such that $\ell \neq p$.
 \end{proposition}

\begin{proof}
Let $\pi_{\wp}$  be the Frobenius automorphism of $\psi(\F_{\wp})$, which may also be viewed as a root of 
$P_{\psi, \wp}$ of Theorem \ref{gekeler1}.
Note that $\Ker(\pi_{\wp} - 1) = \psi(\F_{\wp})$.
 Since $\gcd(m,p) = 1$, Theorem \ref{structure-theorem-torsion-points} tells us that
 $(\psi \otimes \F_{\wp})[m] \simeq_A (A/mA)^r$.
 Therefore, $\psi(\F_{\wp})$ contains an isomorphic copy of $(A/mA)^r$ if and only if
 $(\psi \otimes \F_{\wp})[m] \leq_A \psi(\F_{\wp}) = \Ker(\pi_{\wp} - 1)$.

 If $\sigma_{\wp}$ denotes the Frobenius at  $\wp$ in $K(\psi[m]) / K$, then
 $(\psi \otimes \F_{\wp})[m] \leq_A \Ker(\pi_{\wp} - 1)$
 if and only if
 $\psi[m] \leq_A \Ker \left(\sigma_{\wp} - 1 \right)$.
 This last statement is equivalent to $\sigma_{\wp}$ acting trivially on $\psi[m]$, and hence to
$\wp$ splitting completely in $K(\psi[m]) / K$.

\end{proof}

\begin{proposition}\label{split-divisibility}

Let $K$ be a finite field extension of $k$ and let $\psi \in \Drin_{A}(K)$ be of generic characteristic and 
 rank $r \geq 1$.
Let $\wp \in {\cal{P}}_{\psi}$ and  let $a \in A \backslash \F_q$ be such that $\gcd(a, p) = 1$, where $\wp \cap A = p A$.
If $\wp$ splits completely in $K(\psi[a])$, then $a^r | P_{\psi, \wp}(1)$.
\end{proposition}
\begin{proof}
Let again  $\pi_{\wp}$  be the Frobenius automorphism of $\psi(\F_{\wp)}$.
Since $\wp$ splits completely in $K(\psi[a])$, $\sigma_{\wp}$ acts trivially on $\psi[a]$, and so
$(\psi \otimes \F_{\wp})[d] \leq_A \Ker (\pi_{\wp} -1)$.
Recalling the structure of the torsion of $\psi \otimes \F_{\wp}$, we deduce that  
$\psi(\F_{\wp})$ contains an isomorphic copy of $(A/d A)^r$.  By taking the Euler-Poincar\'{e} characteristic and  by invoking part (v) of Proposition \ref{gekeler1}, we then deduce the desired divisibility relation.
\end{proof}

By combining Proposition \ref{split-divisibility} with the results of \cite{He} providing a Drinfeld module   analoque of the 
Weil pairing for elliptic curves, we obtain: 
\begin{theorem}(Properties of primes splitting completely in division fields)
\label{split-all}

\noindent
Let $K$ be a finite field extension of $k$ and let  $\psi \in \Drin_{A}(K)$ be of generic characteristic and rank $r \geq 2$.  
Then there exists $\psi^{1} \in \Drin_{A}(K)$, of generic characteristic and of  rank 1, uniquely determined up to $\overline{K}$-isomorphism, such that:
\begin{enumerate}
\item[(i)]
${\mathcal{P}}_{\psi} \subseteq \mathcal{P}_{\psi^1}$;
\item[(ii)]
for any $\wp \in {\mathcal{P}}_{\psi}$, the characteristic polynomials of $\psi$ and $\psi^1$ at $\wp$ satisfy the relation:
$$
P_{\psi, \wp}(X) = X^r + a_{r-1, \wp}(\psi) X^{r-1} + \ldots + a_1(\psi, \wp) X + u_{\wp}(\psi)p^{m_{\wp}},
$$
$$
P_{\psi^1, \wp}(X) = X + (-1)^{r-1} u_{\wp}(\psi) p^{m_{\wp}},
$$
where  $u_{\wp}(\psi) \in \F_q^{\ast}$;
\item[(iii)]
for any $\wp \in \mathcal{P}_{\psi}$ and any  $a \in A \backslash \F_{q}$ coprime to $p$, 
where $\wp \cap A = p A$,
 if $\wp$ splits completely in $K(\psi[a])$, then:
\begin{enumerate} 
\item[(iii1)]
$\wp$ also splits completely in $K(\psi^{1}[a])$;
\item[(iii2)]  
$a^r | P_{\psi, \wp}(1)$;
\item[(iii3)]
$a | P_{\psi^{1}, \wp}(1)$.
\end{enumerate}
\end{enumerate}
\end{theorem}
\begin{proof}
See \cite[Proposition 10]{CoSh}.
\end{proof}

\section{The Chebotarev density theorem}

Let $K$ be a finite field extension of $k$ and let $K'/K$ be a finite Galois extension.
In this section, we recall an effective version of the Chebotarev Density Theorem for $K'/K$,
 as proven in \cite{MuSc}.

 Let $g_{K'}$ and $g_K$ be the genera of $K'$ and $K$, respectively, and
let $c_{K'}$ denote the degree of the constant field $\F_{K'}$ of $K'$ over $\F_K$, that is,
$$c_{K'} := [K' \cap \overline{\F}_K : \F_K ].$$
Let
\[D := 
\ds\sum_{ \wp \textrm{ ramified in } K'/K}
\deg_K \wp.
\]

Let $x \in \N$ and set
\[\Pi(x; K'/K) := \#\{\wp \; \textrm{ unramified in } K'/K : \deg_K \wp = x\}.\]

For $C \subseteq \Gal(K'/K)$  a conjugacy class,  set
\[\Pi_C(x; K'/K)
:= \#\{\wp \; \textrm{ unramified in } K'/K : \deg_K \wp = x, \sigma_{\wp} = C\},\]
 where $\sigma_{\wp}$ is the
Frobenius at $\wp$ in $K'/K$. Note that, in particular,
 $\Pi_1(x; K'/K)$ denotes the number of degree $x$ primes of $K$ which split completely in $K'$.
Let $a_C \in \N$ be defined by the property that the restriction of $C$ to $\F_{K'}$  is $\tau^{a_C}$.

\begin{theorem}\label{chebotarev} \cite[Theorem 1, p. 524]{MuSc}

\noindent
We keep the above setting and notation.
\begin{enumerate}
\item[(i)] If $x \not\equiv a_C (\mod c_{K'})$, then $\Pi_C(x;K'/K) = 0$.
\item[(ii)] If $x \equiv a_C (\mod c_{K'})$, then
\[
\left| \Pi_C(x;K'/K) - c_{K'} \frac{|C|}{|G|} \Pi(x;K'/K) \right|
\leq
2 g_{K'} \frac{|C|}{|G|} \frac{q^{\frac{c_K x}{2}}}{x} + 2(2g_K+1) |C| \frac{q^{\frac{c_K x}{2}}}{x} + \left(1 + \frac{|C|}{x} \right) D.
\]
    \end{enumerate}
\end{theorem}

Our main application of Theorem \ref{chebotarev} is when $K'$ is a division field of
a generic Drinfeld module  $\psi \in \Drin_A(K)$ and when  $C= \{1\}$.  For ease of understanding, we record a restatement of this theorem in our desired setting:

\begin{theorem}\label{cheb-dm}
Let $K$ be a finite field extension of $k$ and let $\psi \in \Drin_A(K)$ be of generic characteristic.
Let $a \in A \backslash \F_q$. Let $x \in \N$ and define
 \begin{equation} \label{defn-of-ca(x)}
  c_a(x)
  :=
  \left\{ \begin{array}{cl}
 c_a & \textrm{if } c_a | x, \\
  0  &    \textrm{else},     \\
\end{array}
 \right. \end{equation}
 where $c_a$ denotes the degree of the constant field extension of $K(\psi[a])$ over $\F_K$ 
 (see (\ref{c-a})).
Then
$$
\Pi_1(x; K(\psi[a])/K)
=
\frac{c_a(x)}{[K(\psi[a]) : K]} \cdot  \frac{q^{c_K x}}{x}
+
\O_{\psi, K}\left(\frac{q^{\frac{c_K x}{2}}}{x} \deg a\right).
$$
\end{theorem}
\begin{proof}  Using the effective Prime Number Theorem for $K$ (see (\ref{PNT})) and Theorem \ref{chebotarev} with $K' = K(\psi[a])$, $C = \{1\}$, and hence with 
$a_C = 0$, we obtain
\begin{eqnarray*}
\Pi_1(x;K(\psi[a])/K)
& = &
\frac{c_a(x)}{[K(\psi[a]) : K]} \cdot \frac{q^{c_K x}}{x}
+
 \O\left(
 \left( 2 g_a \cdot \frac{1}{[K(\psi[a]) : K]} + 2(2g_K +1) \right) \cdot \frac{q^{\frac{c_K x}{2}}}{x}
 +
 \left( 1 + \frac{1}{x} \right) D
 \right),
\end{eqnarray*}
where
 \[
 D := \ds\sum_{\wp \in {\cal{P}}_{\psi} \atop{ \wp \textrm{ ramified in } K(\psi[a]) / K}} \deg_K \wp.
 \]
By Theorem \ref{takahashi},   $D \ll_{K} \deg a$.  
 Combining this with  Proposition \ref{genus-proposition}, we obtain
 \begin{eqnarray*}
 \left( 2 g_a \cdot \frac{1}{[K(\psi[a]) : K]} + 2(2g_K +1) \right) \cdot \frac{q^{\frac{c_K x}{2}}}{x}
 +
 \left( 1 + \frac{1}{x} \right) D
 & \ll_K &
 \frac{G(\psi,K) \cdot [K(\psi[a]) : K] \cdot \deg a}{[K(\psi[a]) : K]} \cdot \frac{q^{\frac{c_K x}{2}}}{x}
 \\
 &&
 + \frac{x+1}{x} \cdot \deg a \\ & \ll_K & G(\psi,K) \cdot \frac{q^{\frac{c_K x}{2}}}{x} \cdot \deg a.
 \end{eqnarray*}
\end{proof}


\section{Proof of Theorems \ref{thm1} and \ref{thm2}}

Let $K/k$ be a finite field extension and let $\psi : A \longrightarrow K\{\tau\}$ be a generic Drinfeld 
$A$-module over $K$, of rank $r \geq 2$.
Let $d \in A^{(1)}$.
For a fixed $x \in \N$, let 
$$
\newd(\psi, x) := \#\{\wp \in {\cal{P}}_{\psi} : \deg_K \wp = x,\ d_{1,\wp}(\psi) = d\}.
$$
Our first goal is to derive an asymptotic formula for this function, as $x \rightarrow \infty$.

We start by noting that,  for any prime $\wp \in {\cal{P}}_{\psi}$ such that $\gcd(d,p) = 1$ (where, as usual, $\wp \cap A = pA$), we have 
$d = d_{1,\wp}(\psi)$ if and only if 
$\psi(\F_{\wp}) \geq_A (A / dA)^r$ 
and 
$\psi(\F_{\wp}) \not\geq_A (A / \ell d A)^r$ for all primes $\ell \in A^{(1)}$.  
Hence, by the inclusion-exclusion principle and by Proposition \ref{sc-iff}, 
\begin{equation} \label{inc-exc-result} 
\newd(\psi, x) = 
\ds\sum_{
\substack{m \in A^{(1)} \\ \deg m \leq \frac{c_K x}{r}}
} \mu_A(m) \Pi_1(x;K(\psi[md]) / K),
\end{equation} 
where,
for square-free $m$,  the field extension $K(\psi[md]) / K$  is obtained via field composition
\[
\ds\prod_{\substack{\ell | m \\ \ell \text{ prime}}} K(\psi[\ell d]) = K(\psi[\lcm(\ell d : \ell | m)]) = K(\psi[md]),
\] 
and where, we recall, 
 \[
 \Pi_1(x;K(\psi[md]) / K) 
 := 
 \#\{
 \wp \in {\cal{P}}_{\psi} : \deg_K \wp = x, \wp \textrm{ splits completely in } K(\psi[md]) / K
 \}.
 \]

The range of $\deg m$ in the summation on the right-hand side of (\ref{inc-exc-result})  is derived from
the condition that, if  $\wp$ splits completely in $K(\psi[m d])$, then 
 $$
 (A/m d A)^r \leq_A \psi(\F_{\wp}).
 $$
 This gives  the divisibility relation 
 $$
 m^r d^r | \chi(\psi(\F_{\wp}))
 $$
 obtained by taking Euler-Poincar\'{e} characteristic on both sides.
Indeed,  since $|\chi(\psi(\F_{\wp}))|_{\infty} = |\wp|_{\infty} = q^{c_K x}$, the above gives
 \begin{equation}\label{degree}
 \deg m \leq \frac{c_K x}{r}.
 \end{equation}

 The obvious tool in estimating  $\newd(\psi, x)$ is the effective Chebotarev Density Theorem (Theorem \ref{cheb-dm}). However, for $r = 2$,  this is insufficient. As such, we split the sum into two parts,  apply Theorem \ref{cheb-dm} to the first, and find a different approach for  the second. To be precise, we write 
\begin{equation}\label{d-first-splitting}
\newd(\psi, x) = \newd_1(\psi, x, y) + \newd_2(\psi, x, y)
\end{equation}
for some positive real number $y = y(x)$, to be chosen optimally later, 
where 
\[
\newd_1(\psi, x, y) := 
\ds\sum_{\substack{m \in A^{(1)} \\ \deg m \leq y}}
 \mu_A(m) \Pi_1(x;K(\psi[md]) / K)
 \] 
 and 
 \[
 \newd_2(\psi, x, y) := 
 \ds\sum_{\substack{m \in A^{(1)} \\ y < \deg m \leq \frac{c_K x}{r}}} 
 \mu_A(m) \Pi_1(x;K(\psi[md]) / K).
 \]  
By applying Theorem \ref{cheb-dm} and Lemma \ref{lemma1}, we obtain 
\begin{equation} \label{error-f1} 
\newd_1(\psi, x, y) = 
\frac{q^{c_K x}}{x} 
\sum_{\substack{m \in A^{(1)} \\ \deg m  \leq y}} \frac{\mu_A(m) c_{md}(x)}{[K(\psi[md]) : K]} 
+
 \O_{\psi,K, d} \left(\frac{q^{\frac{c_K x}{2} + y} y}{x} \right).
\end{equation}

Now let 
$$\gamma := \text{rank}_{A}  \; \End_{\overline{K}}(\psi).$$
By  Proposition \ref{constant-field-size-proposition},  Theorem \ref{division-degree}, and Lemma \ref{lemma2} (for which we are also using that $\gamma \leq r$),
we obtain
\begin{equation}\label{tail-d}
\ds\sum_{m \in A^{(1)} \atop{\deg m > y}}
\frac{\mu_A(m) c_{m d}(x)}{[K(\psi[md]) : K]}
\ll_{\psi, K, d}
\ds\sum_{m \in A^{(1)} \atop{\deg m > y}}
\frac{\log \deg (m d) + \log \log q}{q^{\frac{r^2 \deg (m d)}{\gamma}}}
\ll
\frac{\log y}{q^{\left(\frac{r^2}{\gamma} - 1\right) y } }.
\end{equation}
Thus
\begin{equation}\label{estimate-d1}
\newd_1(\psi,  x, y) 
= 
\frac{q^{c_K x} }{x} 
\ds\sum_{m \in A^{(1)} }
 \frac{ \mu_A(m) c_{md}(x) }{ [K(\psi[md]) : K] } 
+
 \O_{\psi,K, d} \left(\frac{   q^{ \frac{c_K x}{2} + y} y}{x}\right)
 +
 \O_{\psi, K}
 \left(
 q^{  c_K x - \left(  \frac{r^2}{\gamma} - 1  \right) y    }  
 \log y
 \right).
\end{equation}
Observe that by choosing $y := \frac{1}{r} c_K x$, if $r \geq 4$, then the above estimate already proves the first part of the theorem. The following discussion thus pertains to the case $r \leq 3$.

To estimate $\newd_2(\psi, x, y)$  from above,  we make use of van der Heiden's construction of the analogue of the Weil pairing for $\psi$, as well as of the average over $m$.
More precisely, we appeal to Theorem \ref{split-all} and rely on the properties of  the rank 1 Drinfeld $A$-module $\psi^1 \in \Drin_A(K)$ associated to $\psi$, as follows. 

By part (iii1) of Theorem \ref{split-all},
if $\wp$ splits completely in $K(\psi[md])/K$, then $\wp$ splits completely in $K(\psi^1[md])/K$.
Using the characteristic polynomials at $\wp$ associated to $\psi$ and $\psi^1$, 
by part (ii) of Theorem  \ref{split-all},  this implies that
$$m^r d^r | 1 + a_{\wp}(\psi) + u_{\wp}(\psi) p^{m_{\wp}}$$
and
$$m d | 1 + (-1)^{r-1} u_{\wp}(\psi) p^{m_{\wp}},$$
where, we recall,
$$
P_{\wp, \psi}(X) = X^r + a_{r-1, \wp}(\psi) X^{r-1} + \ldots + a_{1, \wp}(\psi) X + u_{\wp}(\psi) p^{m_{\wp}} \in A[X]
$$
with $u_{\wp}(\psi) \in \F_q^*$, and where we define
$$
a_{\wp}(\psi) := a_{r-1, \wp}(\psi) + \ldots + a_{1, \wp}(\psi).
$$
By  part (iv) of Theorem \ref{gekeler1}, we obtain that
$$
\deg a_{\wp}(\psi) \leq \frac{(r-1) c_K \deg_K \wp}{r}.
$$
Therefore
\begin{eqnarray*}
\left| \newd_2(\psi, x, y)\right|
 & \leq &
 \ds\sum_{\substack{m \in A^{(1)} \\ y < \deg m \leq \frac{c_K x}{r}}}
 \ds\sum_{u \in \F_q^*}
\ds\sum_{\substack{a \in A \\ \deg a \leq \frac{(r-1) c_K x}{r}}}
\ds\sum_{\substack{\wp \in \PP_{\psi} \\ \deg_K \wp = x 
\\ a_{\wp}(\psi) = a, u_{\wp}(\psi)  = u
\\ m^r d^r | 1 + a_{\wp}(\psi) + u_{\wp}(\psi) p^{m_{\wp}}
 \\ m d | 1 + (-1)^{r-1}
u_{\wp}(\psi) p^{m_{\wp}}}} 1. 
\end{eqnarray*}
To simplify notation, for each $a \in A$ let us define
\begin{equation} \label{defn-of-tilde}
\tilde{a} :=
\left\{ \begin{array}{cc}
2 + a & \textrm{if $r$ even} , \\
 a & \textrm{if $r$ odd}.
\end{array}
 \right.
\end{equation}
Thus
$$
\left|\newd_2(\psi, x, y)\right|
\leq
\ds\sum_{\substack{m \in A^{(1)} \\ y < \deg m \leq \frac{c_K x}{r}}}
\ds\sum_{u \in \F_q^*}
\ds\sum_{\substack{a \in A \\ \deg a \leq \frac{(r-1) c_K x}{r} \\ m d | \tilde{a} }}
\ds\sum_{\substack{\wp \in \PP_{\psi} \\ \deg_K \wp = x \\ 
m^r d^r | 1 + a + u p^{m_{\wp}}}}
1.
$$
We now consider the innermost sum above.
Using diagram (\ref{picture-residue-field}), we see that
$$
\deg p^{m_{\wp}} = m_{\wp} \deg p = c_K \deg_K \wp = c_K x.
$$
We also see that, by part (i) of  Lemma \ref{lemma1},
for fixed $a \in A$ and $u \in \F_q^*$,
there exist at most  $q^{c_K x - r \deg m + 1}$ primes $p \in A$ of degree $\frac{c_K x}{m_{\wp}}$ such that
$m^r | 1 + a + u p^{m_{\wp}}$.
Indeed, this reduces to counting polynomials in $A$ of degree 
$\deg (1 + a + u p^{m_{\wp}}) - r \deg m = c_K x - r \deg m$.
But there are at most $[K:k]$ primes in $K$ lying above a fixed prime $p \in A$.
Therefore
$$
\ds\sum_{\substack{\wp \in \PP_{\psi} \\ \deg_K \wp = x \\ m^r | 1 + a + u
p^{m_{\wp}}}}
1
\ll_K
q^{c_K x - r \deg m + 1}
\ll
q^{c_K x - r \deg m}.
$$
Continuing, we deduce that
\begin{eqnarray*}
\newd_2(\psi, x, y)
& \ll_{K, d} &
\ds\sum_{\substack{m \in A^{(1)} \\ y < \deg m \leq \frac{c_K x}{r}}}
\ds\sum_{u \in \F_q^*}
\ds\sum_{\substack{a \in A \\ \deg a \leq \frac{(r-1) c_K x}{r} \\ m | \tilde{a} }} 
q^{c_K x - r \deg m
}
\\
& = &
q^{c_K x
}
\ds\sum_{\substack{m \in A^{(1)} \\ y < \deg m \leq \frac{c_K x}{r}}} q^{-r \deg m}
\ds\sum_{u \in \F_q^*}
\ds\sum_{\substack{a \in A \\ \deg a \leq \frac{(r-1) c_K x}{r} \\
\tilde{a} \neq 0
\\
m| \tilde{a}
}}
1
\\
&&
+\ q^{c_K x 
}
\ds\sum_{\substack{m \in A^{(1)} \\ y < \deg m \leq \frac{c_K x}{r}}} q^{-r \deg m}
\ds\sum_{u \in \F_q^*}
\ds\sum_{\substack{a \in A \\ \deg a \leq \frac{(r-1) c_K x}{r} \\
\tilde{a} = 0
}}
1
\\
& =: &
\newd_{2,1}(\psi, x, y) + \newd_{2,2}(\psi, x, y).
\end{eqnarray*}
We will estimate these two sums from above using Lemmas \ref{lemma1} and \ref{lemma2}, where, for the latter, we will be implicitly also using  that $\gamma \leq r$.

To estimate $\newd_{2, 1}(\psi, x, y)$ from above,  we note that
$$
\newd_{2, 1}(\psi, x, y)
\leq
q^{c_K x + 1}
\ds\sum_{
m \in A^{(1)}
\atop{
y < \deg m \leq \frac{c_K x}{r}
}
}
q^{- r \deg m}
\ds\sum_{
\alpha \in A
\atop{
\deg \alpha \leq \frac{(r - 1) c_K x}{r} - \deg m
}
}
1,
$$
which, by part (i) of Lemma \ref{lemma1} and part (i) of Lemma \ref{lemma2}, is
$$
\ll
q^{c_K x + 1}
\ds\sum_{
m \in A^{(1)}
\atop{
y < \deg m \leq \frac{c_K x}{r}
}
}
q^{-r \deg m}
q^{\frac{(r - 1) c_K x}{r} - \deg m}
\ll
q^{\frac{(2 r - 1) c_K x}{r} - r y}.
$$

To estimate $\newd_{2, 2}(\psi, x, y)$ from above, we note that its innermost sum has only one term, hence, by part (i) of Lemma \ref{lemma2},
$$
\newd_{2, 2}(\psi, x, y) \ll q^{c_K x - (r - 1) y}.
$$

Combining these estimates,
we obtain that
$$
\newd_{2}(\psi, x, y) \ll_{K, d}
q^{ \frac{(2 r - 1) c_K x}{r} - r y}
+
q^{c_K x - (r-1)y}.
$$
Plugging this back into (\ref{d-first-splitting}) and appealing to  (\ref{estimate-d1}), we  deduce that
\begin{eqnarray}\label{estimateDwithy}
&&
\newd(\psi,  x) 
=
\frac{q^{c_K x} }{x} 
\ds\sum_{m \in A^{(1)} }
 \frac{ \mu_A(m) c_{md}(x) }{ [K(\psi[md]) : K] } 
 \nonumber
 \\
&+&
 \O_{\psi,K, d} \left(\frac{   q^{ \frac{c_K x}{2} + y} y}{x}\right)
 +
 \O_{\psi, K}
 \left(
q^{  c_K x - \left(  \frac{r^2}{\gamma} - 1  \right) y    }  
 \log y
 \right)
 +
 \O_{K, d}\left(q^{ \frac{(2 r - 1) c_K x}{r} - r y}
+
q^{c_K x - (r-1)y}\right).
\end{eqnarray}
Finally, we choose $y$ as follows:
\begin{equation} \label{choice-of-y}
y :=
\left\{ \begin{array}{cc}
\frac{3 r - 2}{2 r (r + 1)} c_K x & \textrm{if $r= 2, 3$}, \\
 \frac{1}{r} c_K x & \textrm{if $r \geq 4$}.
\end{array}
 \right.
\end{equation}
We plug this in (\ref{estimateDwithy}) if $r = 2, 3$, and in (\ref{estimate-d1}) if $r \geq 4$ (note that in this case $\newd(\psi, x) = \newd_1(\psi, x, y)$), obtaining  the effective asymptotic formulae:
\begin{equation}\label{effective-d}
\newd(\psi,  x) 
=
\frac{q^{c_K x} }{x} 
\ds\sum_{m \in A^{(1)} }
 \frac{ \mu_A(m) c_{md}(x) }{ [K(\psi[md]) : K] }  
 +
 \left\{ 
 \begin{array}{ccc}
\O_{\psi, K, d}\left(q^{\frac{5 c_K x}{6}}\right) & \textrm{if $r=2$}, \\
\O_{\psi, K, d}\left(q^{\frac{19 c_K x}{24}}\right) & \textrm{if $r=3$}, \\
\O_{\psi, K, d}\left(q^{\frac{(r+2) c_K x}{2 r}}\right) & \textrm{if $r \geq 4$}. 
\end{array}
 \right.
\end{equation}
This completes the first part of Theorem \ref{thm1}.

\begin{remark}
\noindent
\begin{enumerate}
\item[(i)]
Formula (\ref{effective-d}) is stronger than the asymptotic formula  (\ref{formula-thm1}) stated in Theorem \ref{thm1}, as it provides us with explicit error terms. Moreover, these error terms carry significant savings in powers of $q^{c_K x}$.
\item[(ii)]
For $r \geq 4$, the splitting of $\newd(\psi,  x)$ into two parts, as in (\ref{d-first-splitting}), is unnecessary. The proof of Theorem \ref{thm1} in this case is solely an application of the effective Chebotarev Density Theorem for the division fields of $\psi$. 
\item[(iii)]
For $r = 3$, the splitting of $\newd(\psi,  x)$ into two parts, as in (\ref{d-first-splitting}),  is also unnecessary in order to obtain the asymptotic formula (\ref{formula-thm1}). In our proof, we do so 
in order  to obtain a saving in the final error term:
$\O_{\psi, K, d}\left(q^{\frac{19 c_K x}{24}}\right)$ using  (\ref{d-first-splitting}) and the approach therein for estimating $\newd_2(\psi, x, y)$,
versus
$\O_{\psi, K, d}\left(q^{\frac{5 c_K x}{6}}\right)$  using only the effective Chebotarev Density Theorem 
(and the choice $y  := \frac{1}{3} c_K x$).
\item[(iv)]
The error terms in (\ref{effective-d}) may be improved with additional techniques. For example, for the case $q$ odd,  $r = 2$,  $\gamma = 2$,  and $ K = k$, in \cite{CoSh} we proved the formula
\begin{equation*}
\newd(\psi,  x) 
=
\frac{q^{x} }{x} 
\ds\sum_{m \in A^{(1)} }
 \frac{ \mu_A(m) c_{md}(x) }{ [k(\psi[md]) : k] }  
 +
\O_{\psi, k, d}\left(q^{\frac{3  x}{4}}\right). 
\end{equation*}
\end{enumerate}
\end{remark}

Our second goal is to prove that the Dirichlet density of the set
$
\{\wp \in {\cal{P}}_{\psi}: d_{1, \wp}(\psi) = d\}
$
exists and equals 
$
\ds\sum_{
m \in A^{(1)}
}
\frac{\mu_A(m)}{[K(\psi[m d]) : K]}.
$
For this, let $s > 1$ and consider the sum
\begin{eqnarray}\label{dirichlet1}
\ds\sum_{
\wp \in {\cal{P}}_{\psi}
\atop{
d_{1, \wp}(\psi) = d
}
}
q^{- s c_K \deg_K \wp}
&=&
\ds\sum_{x \geq 1}
q^{- s c_K x} \newd(\psi, x)
\nonumber
\\
&=&
\ds\sum_{m \in A^{(1)}}
\frac{\mu_A(m)}{[K(\psi[m d]) : K]}
\ds\sum_{x \geq 1}
\frac{q^{(1-s) c_K x} c_{m d}(x)}{x}
+
\O_{\psi, K}\left(
\ds\sum_{x \geq 1} q^{(\theta(r) - s) c_K x}
\right),
\end{eqnarray}
where we used (\ref{effective-d}) with
\begin{equation}
\theta(r) :=
 \left\{ 
 \begin{array}{ccc}
\frac{5}{6} & \textrm{if $r=2$}, \\
\frac{19}{24} & \textrm{if $r=3$}, \\
\frac{r+2}{2 r} & \textrm{if $r \geq 4$}. 
\end{array}
 \right.
\end{equation}

By the definition  (\ref{defn-of-ca(x)}) of $c_{m d}(x)$, (\ref{dirichlet1}) becomes
$$
=
\ds\sum_{m \in A^{(1)}}
\frac{\mu_A(m)}{[K(\psi[m d]) : K]}
\ds\sum_{j \geq 1}
\frac{q^{(1-s) c_K c_{m d} j}}{j}
+
\O_{\psi, K}\left(
\ds\sum_{x \geq 1} q^{(\theta(r) - s) c_K x}
\right).
$$
Since $s > 1$, this may be written as
$$
=
- \ds\sum_{m \in A^{(1)}}
\frac{\mu_A(m)}{[K(\psi[m d]) : K]}
\log \left(1 - q^{(1 - s) c_K c_{m d}} \right)
+
\O_{\psi, K}
\left(\frac{q^{(\theta(r) - s) c_K} }{1 - q^{(\theta(r) - s) c_K} }\right).
$$

We now calculate
\begin{eqnarray*}
\ds\lim_{s \rightarrow 1+}
\frac{
\ds\sum_{
\wp \in {\cal{P}}_{\psi}
\atop{
d_{1, \wp}(\psi) = d
}
}
q^{- s c_K \deg_K \wp}
}{- \log \left(1 - q^{(1 - s) c_K}\right)}
&=&
\ds\lim_{s \rightarrow 1+}
\left[
\ds\sum_{m \in A^{(1)}}
\frac{\mu_A(m)}{[K(\psi[m d]) : K]}
\frac{
\log \left(1 - q^{(1-s) c_K c_{m d}}\right)
}{
\log \left(1 - q^{(1-s) c_K}\right)
} 
\right.
\\
&&
+
\left.
\O_{\psi, K}
\left(
\frac{
q^{(\theta(r) - s) c_K}
}{
\left(1 - q^{(\theta(r) - s) c_K}\right)
\log \left(1 - q^{(1-s) c_K}\right)
}
\right)
\right]
\\
&=&
\ds\sum_{m \in A^{(1)}} \frac{\mu_A(m)}{[K(\psi[m d]) : K]},
\end{eqnarray*}
after an application of l'Hospital and elementary manipulations. This completes the proof of Theorem \ref{thm1}.

The proof of Theorem \ref{thm2} proceeds in the same way as that of the first part of Theorem \ref{thm1},
after replacing with 1 the factor $\mu_A(m)$ appearing in (\ref{inc-exc-result}), and, henceforth, in  
$\newd_1(\psi, x, y)$ and $\newd_2(\psi, x, y)$ of (\ref{d-first-splitting}). 
In particular, this approach leads to the asymptotic formulae:
\begin{eqnarray}\label{effective-thm2}
\ds\sum_{m \in A^{(1)}}
\Pi_1(x, K(\psi[m])/K)
&=&
\ds\sum_{
m \in A^{(1)}
\atop{
\deg m \leq \frac{c_K x}{r}
}
}
\Pi_1(x, K(\psi[m])/K)
\nonumber
\\
&=&
\frac{q^{c_K x} }{x} 
\ds\sum_{m \in A^{(1)} 
\atop{\deg m \leq \frac{c_K x}{r}}
}
 \frac{c_{m}(x) }{ [K(\psi[m]) : K] }  
 +
 \left\{ 
 \begin{array}{ccc}
\O_{\psi, K}\left(q^{\frac{5 c_K x}{6}}\right) & \textrm{if $r=2$}, \\
\O_{\psi, K}\left(q^{\frac{19 c_K x}{24}}\right) & \textrm{if $r=3$}, \\
\O_{\psi, K}\left(q^{\frac{(r+2) c_K x}{2 r}}\right) & \textrm{if $r \geq 4$}
\end{array}
 \right.
\\
&=&
\frac{q^{c_K x} }{x} 
\ds\sum_{m \in A^{(1)} }
 \frac{c_{m}(x) }{ [K(\psi[m]) : K] }  
 +
 \left\{ 
 \begin{array}{ccc}
\O_{\psi, K}\left(q^{\frac{5 c_K x}{6}}\right) & \textrm{if $r=2$}, \\
\O_{\psi, K}\left(q^{\frac{19 c_K x}{24}}\right) & \textrm{if $r=3$}, \\
\O_{\psi, K}\left(q^{\frac{(r+2) c_K x}{2 r}}\right) & \textrm{if $r \geq 4$}. 
\end{array}
 \right.
 \nonumber
\end{eqnarray}

\section{Proof of Theorem \ref{thm3}}

Let $K/k$ be a finite field extension and let $\psi : A \longrightarrow K\{\tau\}$ be a generic Drinfeld $A$-module over $K$, of rank 2. Let $\gamma := \text{rank}_{A}  \; \End_{\overline{K}}(\psi)$.

(i) Let $f : (0, \infty) \longrightarrow (0, \infty)$ be such that $\ds\lim_{x \rightarrow \infty} f(x) = \infty$.
For a fixed $x \in \N$, let
$$
{\mathcal{E}}(\psi,  x) 
=
{\mathcal{E}}_f(\psi,  x) 
 :=
\#\left\{
\wp \in {\cal{P}}_{\psi}: \deg_K \wp = x, |d_{2, \wp}(\psi)|_{\infty} > \frac{|\wp|_{\infty}}{q^{c_K f(x)}}
\right\}
$$
and
$$
{\it{e}}(\psi, x) 
=
{\it{e}}_f(\psi, x)
:=
\#\left\{
\wp \in {\cal{P}}_{\psi}: \deg_K \wp = x, |d_{2, \wp}(\psi)|_{\infty} < \frac{|\wp|_{\infty}}{q^{c_K f(x)}}
\right\}.
$$
Our goal is to derive an asymptotic formula for ${\mathcal{E}}(\psi,  x)$, as $x \rightarrow \infty$.
More precisely, our goal is to show that ${\mathcal{E}}(\psi,  x) \sim \pi_K(x)$, which is equivalent to 
showing that
$$
{\it{e}}(\psi,  x) = \o\left(\pi_K(x)\right).
$$

Note that, without loss of generality, we may assume that $f(x) < \frac{x}{2}$ for all $x$. Indeed, if 
$f_1, f_2 : (0, \infty) \longrightarrow (0, \infty)$ satisfy $f_2 \leq f_1$ and
 $\ds\lim_{x \rightarrow \infty} f_2(x) = \infty$, then
 $ {\mathcal{E}}_{f_2}(\psi,  x) 
 \leq
 {\mathcal{E}}_{f_1}(\psi,  x) 
 \leq
 \pi_K(x)$.
 Thus, for the purpose of proving part (i) of Theorem \ref{thm3}, we may 
 replace $f(x)$ by $\min \left\{f(x), \frac{x}{2} - 1\right\}$.

We start with the partition
$$
{\it{e}}(\psi, x) =
\ds\sum_{d \in A^{(1)}}
\ds\sum_{
\wp \in {\mathcal{P}}_{\psi}
\atop{
\deg_K \wp = x
\atop{
d_{1, \wp}(\psi) = d
\atop{
|d_{2, \wp}(\psi)|_{\infty} < \frac{|\wp|_{\infty}}{q^{c_K f(x)}}
}
}
}
}
1,
$$
and remark that, as in the deduction of  (\ref{inc-exc-result}) in the proof of Theorem \ref{thm1}, the condition $d | d_{1, \wp}(\psi)$
imposes the restriction $\deg d \leq \frac{c_K x}{2}$.
Moreover, the conditions $d_{1, \wp}(\psi) = d$ and 
$|d_{2, \wp}(\psi)|_{\infty} < \frac{|\wp|_{\infty}}{q^{c_K f(x)}}$, coupled with the remark 
\begin{equation}\label{euler-char-d1d2}
|\wp|_{\infty} = |\chi(\psi(\F_{\wp}))|_{\infty} = |d_{1, \wp}(\psi)|_{\infty} |d_{2, \wp}(\psi)|_{\infty},
\end{equation}
impose the restriction $c_K f(x) < \deg d$. Thus
\begin{equation*}
{\it{e}}(\psi,  x)
\leq
\ds\sum_{
d \in A^{(1)}
\atop{
c_K f(x) < \deg d \leq \frac{c_K x}{2}
}
}
\#\{\wp \in {\mathcal{P}}_{\psi}: \deg_K \wp = x, d | d_{1, \wp}(\psi)\}
=
\ds\sum_{
d \in A^{(1)}
\atop{
c_K f(x) < \deg d \leq \frac{c_K x}{2}
}
}
\Pi_{1}(x, K(\psi[d])/K),
\end{equation*}
by also using Proposition \ref{sc-iff}. Using version (\ref{effective-thm2}) of Theorem \ref{thm2} for the range $\deg d \leq \frac{c_K x}{2}$
and Theorem \ref{cheb-dm} for the range $\deg d \leq c_K f(x)$, the above is
\begin{eqnarray}\label{effective-e}
=
\frac{ q^{c_K x} }{x}
\ds\sum_{
d \in A^{(1)}
\atop{
c_K f(x) < \deg d \leq \frac{c_K x}{2}
}
}
\frac{c_d(x)}{[K(\psi[d]) : K]}
+
\O_{\psi, K}\left(q^{ \frac{5 c_K x}{6} }\right)
+
\O_{\psi, K}\left(
\frac{q^{\frac{c_K x}{2} + c_K f(x)}  f(x)}{x}
\right).
\end{eqnarray}
Reasoning as for (\ref{tail-d}) in the proof of Theorem \ref{thm1} and using that now $\gamma \leq 2$, the first term becomes
\begin{equation*}
\ll_{\psi, K} 
q^{c_K \left(x - \left(\frac{4}{\gamma} - 1\right) f(x)\right)}
\frac{\log f(x) }{x}.
\end{equation*}
Since $\ds\lim_{x \rightarrow \infty} f(x) = \infty$ and $f(x) < \frac{x}{2}$,
we deduce that ${\it{e}}(\psi,  x) = \o(\pi_K(x))$.

 \begin{remark}
 The same proof can be carried through in the case $r = 3$, leading to a  weaker  result.
 For $r \geq 4$, however, one requires a more careful analysis that also takes into consideration the behaviour of the intermediate  elementary divisors.
 \end{remark}
 
 Reasoning as at the beginning of the proof,  without loss of generality we may assume that
 there is $0 < \theta < 1$ such that $f(x) \leq \frac{\theta x}{2}$ for all $x$.
 We now show that
 the Dirichlet density of the set
 $
 \left\{
 \wp \in {\cal{P}}_{\psi}:
 |d_{2, \wp}|_{\infty} < \frac{|\wp|_{\infty}}{q^{c_K f(\deg_K \wp)}}
 \right\}
 $
 equals $0$.
 
 Let $s > 1$ and consider the sum
 \begin{eqnarray}\label{T}
 &&
 \ds\sum_{
 \wp \in {\cal{P}}_{\psi}
 \atop{
 |d_{2, \wp}(\psi)|_{\infty} < \frac{|\wp|_{\infty}}{q^{c_K f(\deg_K \wp)}}
 }
 }
 q^{- s c_K \deg_K \wp}
 =
 \ds\sum_{x \geq 1} q^{- s c_K x} \it{e}(\psi, x)
 \nonumber
 \\
 &\ll&
 \ds\sum_{x \geq 1}
 \frac{q^{(1-s) c_K x}}{x}
 \ds\sum_{
 d \in A^{(1)}
 \atop{
 c_K f(x) < \deg d \leq \frac{c_K x}{2}
 }
 }
 \frac{c_d(x)}{[K(\psi[d]) : K]}
 +
 \ds\sum_{x \geq 1} q^{\left(\frac{5}{6} - s\right) c_K x}
 +
 \ds\sum_{x \geq 1} \frac{q^{\left(\frac{1}{2} - s\right) c_K x + c_K f(x)} f(x)}{x}
 \nonumber
 \\
&=:&
 T_1 + T_2 + T_3;
 \end{eqnarray}
 here we  have also used the prior estimate (\ref{effective-e}).
 
First  we  focus on $T_1$. By the definition of $c_d(x)$, we obtain
\begin{eqnarray}\label{T1}
T_1
& :=&
\ds\sum_{x \geq 1}
\frac{q^{(1-s) c_K x}}{x}
\ds\sum_{
d \in A^{(1)}
\atop{
c_K f(x) < \deg d \leq \frac{c_K x}{2}
}
}
\frac{c_d(x)}{[K(\psi[d]) : K]}
\nonumber
\\
&=&
\ds\sum_{d \in A^{(1)}}
\ds\sum_{
j \geq 1
\atop{
c_K f(c_d j) < \deg d \leq \frac{c_K c_d j}{2}
} 
}
\frac{q^{(1-s) c_K c_d j}}{j [K(\psi[d]) : K]}.
\end{eqnarray}

Let $M > 0$. Since $\ds\lim_{x \rightarrow \infty} f(x) = \infty$, there exists $n(M) \in \N$ such that
$f(n) > M$ for all $n \geq n(M)$.
We split the inner sum in  (\ref{T1}) according to whether $c_d j \geq n(M)$ and
$c_d j < n(M)$, and consider each of the two emerging sums separately.

By the above and Theorem \ref{division-degree}, we have
\begin{eqnarray*}
T_{1, 1}
&:=&
\ds\sum_{d \in A^{(1)}}
\ds\sum_{
j \geq 1
\atop{
c_d j \geq n(M)
\atop{
c_K f(c_d j) < \deg d \leq \frac{c_K c_d j}{2}
}
}
}
\frac{q^{(1-s) c_K c_d j}}{j [K(\psi[d]) : K]}
\\
&\leq&
\ds\sum_{d \in A^{(1)}}
\ds\sum_{
j \geq 1
\atop{
\atop{
c_K M < \deg d \leq \frac{c_K c_d j}{2}
}
}
}
\frac{q^{(1-s) c_K c_d j}}{j [K(\psi[d]) : K]}
\\
&\ll_{\psi, K}&
\ds\sum_{
d \in A^{(1)}
\atop{
c_K M < \deg d
}
}
\frac{\log \deg d}{|d|_{\infty}^{\frac{4}{\gamma}}}
\ds\sum_{
j \geq 1
\atop{
}
}
\frac{q^{(1-s) c_K c_d j}}{j}
\\
&\leq&
\ds\sum_{
d \in A^{(1)}
\atop{
c_K M < \deg d
}
}
\frac{\log \deg d}{|d|_{\infty}^{\frac{4}{\gamma}}}
\ds\sum_{
j \geq 1
\atop{
}
}
\frac{q^{(1-s) c_K  j}}{j} (CHECK).
\end{eqnarray*}
Since $s > 1$, the  latter becomes
$$
=
-
\ds\sum_{
d \in A^{(1)}
\atop{
c_K M < \deg d
}
}
\frac{\log \deg d}{|d|_{\infty}^{\frac{4}{\gamma}}}
\log \left(1 - q^{(1-s) c_K c_d}\right).
$$ 
By Lemma \ref{lemma2}, this is
$$
\ll
\frac{
\log (c_K M) 
}{
q^{ \left(\frac{4}{\gamma} - 1 \right) c_K M}
\log q
}
\left|
\log \left(1 - q^{(1-s) c_K}\right)
\right|,
$$ 
which gives that
\begin{equation}\label{T11}
\ds\lim_{s \rightarrow 1+}
\frac{T_{1, 1}}{- \log \left(1 - q^{(1-s) c_K}\right)}
\ll
\frac{\log (c_K M)}{q^{\left(\frac{4}{\gamma} - 1 \right) c_K M} \log q}.
\end{equation}

We also have
\begin{eqnarray*}
T_{1, 2}
&:=&
\ds\sum_{d \in A^{(1)}}
\ds\sum_{
j \geq 1
\atop{
c_d j < n(M)
\atop{
c_K f(c_d j) < \deg d \leq \frac{c_K c_d j}{2}
}
}
}
\frac{q^{(1-s) c_K c_d j}}{j [K(\psi[d]) : K]},
\end{eqnarray*}
a finite sum. Since
$
\ds\lim_{s \rightarrow 1+} \frac{q^{(1-s) c_K \alpha}}{\log \left(1 - q^{(1-s) c_K}\right) } = 0
$
for any $\alpha \in \N$, we deduce that
\begin{equation}\label{T12}
\ds\lim_{s \rightarrow 1+} \frac{T_{1, 2}}{- \log \left(1 - q^{(1-s) c_K}\right)} = 0.
\end{equation}
By taking $M \rightarrow \infty$ and by using (\ref{T11}) and (\ref{T12}), we obtain that
\begin{equation}\label{T1-limit}
\ds\lim_{s \rightarrow 1+} \frac{T_1}{- \log \left(1 - q^{(1-s) c_K}\right)} = 0.
\end{equation} 

We now focus on $T_2$ and note that
\begin{eqnarray}\label{T2-limit}
\ds\lim_{s \rightarrow 1+} \frac{T_{2}}{- \log \left(1 - q^{(1-s) c_K}\right)} 
=
-
\ds\lim_{s \rightarrow 1+}
\frac{
q^{\left(\frac{5}{6} - s\right) c_K}
}{
\left(1 - q^{\left(\frac{5}{6} - s\right) c_K}\right)
\log \left(1 - q^{(1-s) c_K}\right)
}
=
0.
\end{eqnarray} 

It remains to focus on $T_3$.  Recalling that now we are assuming that there exists $0 < \theta < 1$ such that $f(x) \leq \frac{\theta x}{2} \; \; \forall x$, we see that
\begin{eqnarray}\label{T3-limit}
\ds\lim_{s \rightarrow 1+} \frac{T_{3}}{- \log \left(1 - q^{(1-s) c_K}\right)} 
&=&
-\ds\lim_{s \rightarrow 1+} 
\frac{
\ds\sum_{x \geq 1} \frac{q^{\left(\frac{1}{2} - s\right) c_K x + c_K f(x)} f(x)}{x}
}{
\log \left(1 - q^{(1-s) c_K}\right)
}
\nonumber
\\
&\ll&
\left|
\ds\lim_{s \rightarrow 1+}
\frac{
\ds\sum_{x \geq 1} q^{\left(\frac{1}{2} + \theta - s\right)} 
}{
\log \left(1 - q^{(1-s) c_K}\right)
}
\right|
\nonumber
\\
&=&
\left|
\ds\lim_{s \rightarrow 1+}
\frac{
q^{\left(\frac{1}{2} + \theta - s\right) c_K} \left(1 - q^{\left(\frac{1}{2} + \theta - s\right) c_K} \right)^{-1}
}{
\log \left(1 - q^{(1-s) c_K}\right)
}
\right|
\nonumber
\\
&=&
0.
\end{eqnarray}

By combining  (\ref{T}) with (\ref{T1-limit}) - (\ref{T3-limit}), we obtain
$$
\ds\lim_{s \rightarrow 1+}
\frac{
\ds\sum_{
\wp \in {\cal{P}}_{\psi}
\atop{
|d_{2, \wp}(\psi)|_{\infty} < \frac{|\wp|_{\infty}}{q^{c_K f(\deg_K \wp)}}
}
}
q^{-s c_K \deg_K \wp}
}{
- \log \left(1 - q^{(1-s) c_K }\right)
}
=
0.
$$
This completes the proof of the first part of Theorem \ref{thm3}. 

(ii) In what follows, we investigate the average of $|d_{2, \wp}(\psi)|_{\infty}$ as $\wp \in \cal{P}_{\psi}$
varies over primes with $\deg_K \wp = x$. 
By using (\ref{euler-char-d1d2}) and Lemma \ref{lemma3},  and by partitioning the primes $\wp$ according to the divisors of $d_{1, \wp}(\psi)$, we obtain:
\begin{eqnarray}
\cal{A}(\psi, x) 
&:=&
\frac{1}{q^{c_K x}}
\ds\sum_{
\wp \in \cal{P}_{\psi}
\atop{
\deg_K \wp = x
}
}
|d_{2, \wp}(\psi)|_{\infty}
\nonumber
\\
&=&
\frac{1}{q^{c_K x}}
\ds\sum_{
\wp \in \cal{P}_{\psi}
\atop{
\deg_K \wp = x
}
}
\frac{|\wp|_{\infty}}{|d_{1, \wp}(\psi)|_{\infty}}
\nonumber
\\
&=&
\ds\sum_{
\wp \in \cal{P}_{\psi}
\atop{
\deg_K \wp = x
}
}
\frac{1}{|d_{1, \wp}(\psi)|_{\infty}}
\nonumber
\\
&=&
\frac{1}{(q-1)^2}
\ds\sum_{
\wp \in \cal{P}_{\psi}
\atop{
\deg_K \wp = x
}
}
\ds\sum_{
a, b \in A
\atop{
a b | d_{1, \wp}(\psi)
}
}
\frac{\mu_A(a)}{|b|_{\infty}}
\nonumber
\\
&=&
\ds\sum_{
\wp \in \cal{P}_{\psi}
\atop{
\deg_K \wp = x
}
}
\ds\sum_{
a, b \in A^{(1)}
\atop{
a b | d_{1, \wp}(\psi)
}
}
\frac{\mu_A(a)}{|b|_{\infty}}
\nonumber
\\
&=&
\ds\sum_{m \in A^{(1)}}
\ds\sum_{
\wp \in \cal{P}_{\psi}
\atop{
\deg_K \wp = x
\atop{
m | d_{1, \wp}(\psi)
}
}
}
\ds\sum_{
a, b \in A^{(1)}
\atop{
a b | d_{1, \wp}(\psi)
}
}
\frac{\mu_A(a)}{|b|_{\infty}}
\nonumber
\\
&=&
\ds\sum_{
m \in A^{(1)}
\atop{
\deg m \leq \frac{c_K x}{2}
}
}
\ds\sum_{
a, b \in A^{(1)}
\atop{
a b |m
}
}
\frac{\mu_A(a)}{|b|_{\infty}}
\Pi_1(x, K(\psi[m])/K),
\end{eqnarray}
where in the last line we used Proposition \ref{sc-iff} and the same derivation as (\ref{degree}) for tghe length of the sum over $m$.

Now we proceed similarly to the study of $\newd(\psi, x)$ in the proof of Theorem \ref{thm1}.
Specifically, we let $$y := \frac{c_K x}{3}$$ and write
\begin{eqnarray}\label{average-split}
\cal{A}(\psi, x)
&=&
\ds\sum_{
m \in A^{(1)}
\atop{
\deg m \leq y
}
}
\ds\sum_{
a, b \in A^{(1)}
\atop{
a b |m
}
}
\frac{\mu_A(a)}{|b|_{\infty}}
\Pi_1(x, K(\psi[m])/K)
+
\ds\sum_{
m \in A^{(1)}
\atop{
y < \deg m \leq \frac{c_K x}{2}
}
}
\ds\sum_{
a, b \in A^{(1)}
\atop{
a b |m
}
}
\frac{\mu_A(a)}{|b|_{\infty}}
\Pi_1(x, K(\psi[m])/K)
\nonumber
\\
&=:&
\cal{A}_1(\psi, x, y) + \cal{A}_2 (\psi, x, y).
\end{eqnarray}

By Theorem \ref{cheb-dm},
\begin{eqnarray*}
\cal{A}_1(\psi, x, y)
&=&
\frac{q^{c_K x}}{x}
\ds\sum_{
m \in A^{(1)}
\atop{\deg m \leq y}
}
\frac{c_m(x)}{[K(\psi[m]) : K]}
\ds\sum_{
a, b \in A^{(1)}
\atop{a b | m}
}
\frac{\mu_A(a)}{|b|_{\infty}}
+
\O_{\psi, K}\left(
\frac{q^{\frac{c_K x}{2}}}{x}
\ds\sum_{m \in A^{(1)}
\atop{\deg m \leq y}
}
\deg m
\ds\sum_{
a, b \in A^{(1)}
\atop{a \; \text{squarefree}
\atop{
a b | m
}
}
}
\frac{1}{|b|_{\infty}}
\right).
\end{eqnarray*}

By Lemma \ref{moebius-sum}  and the observation that
\begin{equation}\label{phi-quotient}
\frac{\phi_A(\rad(m))}{|m|_{\infty}} \leq 1,
\end{equation}
the sum in the $\O$-term of $\cal{A}_1(\psi, x, y)$ becomes
$$
\ll
\ds\sum_{
m \in A^{(1)}
\atop{\deg m \leq y}
}
\frac{\phi_A(\rad (m))}{|m|_{\infty}} \deg m
\ll
q^y y,
$$
upon also using part (ii) of  Lemma \ref{lemma1}.
Thus, recalling our choice of $y$, we obtain that
\begin{equation*}
\cal{A}_1(\psi, x, y)
=
\frac{q^{c_K x}}{x}
\ds\sum_{
m \in A^{(1)}
\atop{
\deg m \leq \frac{c_K x}{3}
}
}
\frac{c_m(x)}{[K(\psi[m]) : K]}
\ds\sum_{
a, b \in A^{(1)}
\atop{a b | m}
}
\frac{\mu_A(a)}{|b|_{\infty}}
+
\O_{\psi, K}
\left(
q^{\frac{5 c_K x}{6}}
\right).
\end{equation*}

To extend the range of $\deg m$ in the  above sum over $m$, we use 
Proposition \ref{constant-field-size-proposition},
Theorem \ref{division-degree},
Lemma \ref{moebius-sum},
(\ref{phi-quotient}),
and part (ii) of Lemma \ref{lemma2}. We obtain: 
\begin{eqnarray*}
\frac{q^{c_K x}}{x}
\left|
\ds\sum_{
m \in A^{(1)}
\atop{
\deg m > y
}
}
\frac{c_m(x)}{[K(\psi[m]) : K]}
\ds\sum_{
a, b \in A^{(1)}
\atop{
a b | m
}
}
\frac{\mu_A(a)}{|b|_{\infty}}
\right|
&\ll_{\psi}&
\frac{q^{c_K x}}{x}
\ds\sum_{
m \in A^{(1)}
\atop{\deg m > y}
}
\frac{1}{[K(\psi[m]) : K]}
\ds\sum_{
a, b \in A^{(1)}
\atop{
a \; \text{squarefree}
\atop{
a b = d
}
}
}
\frac{1}{|b|_{\infty}}
\\
&\ll_{\psi}&
\frac{q^{c_K x}}{x}
\ds\sum_{
m \in A^{(1)}
\atop{
\deg m > y
}
}
\frac{\log \deg m}{|m|_{\infty}^{\frac{4}{\gamma}}}
\cdot
\frac{\phi_A(\rad (m))}{|m|_{\infty}}
\\
&\ll_{\psi}&
\frac{q^{c_K x - \left(\frac{4}{\gamma} - 1\right) y} \log y}{x y},
\end{eqnarray*}
using once again that $\gamma \leq 2$. Recalling that $y = \frac{c_K x}{3}$, we deduce that
\begin{equation}
\cal{A}_1(\psi, x, y)
=
\frac{q^{c_K x}}{x}
\ds\sum_{
m \in A^{(1)}
}
\frac{c_m(x)}{[K(\psi[m]) : K]}
\ds\sum_{
a, b \in A^{(1)}
\atop{a b | m}
}
\frac{\mu_A(a)}{|b|_{\infty}}
+
\O_{\psi, K}
\left(
q^{\frac{5 c_K x}{6}}
\right).
\end{equation}

We now turn to $\cal{A}_2(\psi, x, y)$. By Lemma  \ref{moebius-sum} and  (\ref{phi-quotient}),
we have
$$
|\cal{A}_2(\psi, x, y)|
\leq
\ds\sum_{
m \in A^{(1)}
\atop{
y < \deg m \leq \frac{c_K x}{2}
}
}
\ds\sum_{
a, b \in A^{(1)}
\atop{
a \; \text{squarefree}
\atop{
a b | m
}
}
}
\frac{1}{|b|_{\infty}}
\Pi_1(x, K(\psi[m])/K)
\leq
\ds\sum_{
m \in A^{(1)}
\atop{
y < \deg m \leq \frac{c_K x}{2}
}
}
\Pi_1(x, K(\psi[m])/K).
$$
This is estimated exactly as $|\newd_2(\psi, x, y)|$ in the proof of Theorem \ref{thm1}, giving
the upper bound 
$\O_{\psi, K}\left(q^{\frac{5 c_K x}{6}}\right)$ (with $y = \frac{c_K x}{3}$).

Putting everything together, we deduce that
\begin{equation}\label{effective-average}
\cal{A}(\psi, x)
=
\frac{q^{c_K x}}{x}
\ds\sum_{
m \in A^{(1)}
}
\frac{c_m(x)}{[K(\psi[m]) : K]}
\ds\sum_{
a, b \in A^{(1)}
\atop{a b | m}
}
\frac{\mu_A(a)}{|b|_{\infty}}
+
\O_{\psi, K}\left(q^{\frac{5 c_K x}{6}}\right),
\end{equation}
completing the proof of part (ii) of Theorem \ref{thm3}.

\begin{remark}
As with Theorem \ref{thm1},  (\ref{effective-average}) is stronger than the asymptotic formula stated in 
part (ii) of Theorem \ref{thm3}, as it provides us with explicit error terms. 
Moreover, when $q$ odd, $r = 2$, $\gamma = 2$, and $K = k$, the  methods of \cite{CoSh} lead to the improved formula
\begin{eqnarray*}
\frac{1}{q^{x}}
\ds\sum_{
\wp \in {\cal{P}}_{\psi}
\atop{\deg \wp = x}
}
|d_{2, \wp}(\psi)|_{\infty}
&=&
\frac{q^{x}}{x}
\ds\sum_{d \in A^{(1)}} \frac{c_d(x)}{[k(\psi[d]) : k]}
\ds\sum_{m, n \in A^{(1)} \atop{m n = d}} \frac{\mu_A(m)}{|n|_{\infty}}
+
\O_{\psi, k}
\left(
q^{\frac{3  x}{4}}
\right).
\end{eqnarray*}
\end{remark}


\section{Concluding remarks}

\noindent
\begin{remark}
{\emph{
Before investigating the Drinfeld module analogues of (\ref{duke}) and (\ref{cojocaru-murty}),  it is natural to consider their analogues in the context of an elliptic curve $E$ over a global function field. 
This is precisely the content of \cite{CoTo}.
For clarity, one of the results  Cojocaru and T\'{o}th prove is that,  given an
elliptic curve $E$ over $K := \F_q(T)$ with $j(E) \not\in \F_q$, then,  provided $\text{char} \; \F_q \geq 5$,
for any $x \in \N$ such that
$x \rightarrow \infty$ and for any $\varepsilon > 0$, we have
\begin{eqnarray}\label{coto}
\#\{\fp \; \text{prime of good reduction for $E$}:  
\deg (\fp) = x,
E_{\fp}(\F_{\fp}) \; \text{cyclic} \} 
&=&
\ds\sum_{
m \geq 1 
\atop{
(m, \text{char} \; \F_q) = 1
\atop{m | q^x - 1
}
}
}
\frac{\mu(m) c_{m}}{[K(E[m]): K]}
\cdot
 \frac{q^x}{x}
 \nonumber
\\
&+&
\O_{E, \varepsilon}
\left(
 \frac{q^{x \left(\frac{1}{2} + \varepsilon\right)}  }{x}
 \right),
 \end{eqnarray}
  where
$\mu(\cdot)$ is the M\"{o}bius function on $\Z$, 
$K(E[m])$ is the $m$th division field of $E$,  
and $c_{m}$ is the multiplicative order of $q$ modulo $m$.
The   formula  is  unconditional. Unusually, it is a direct  consequence of the effective Chebotarev Density Theorem for function fields, no extra sieving  being required. 
This special simplification occurs thanks to the inclusion 
$K \F_{q^{c_m}} \subseteq K(E[m])$, 
and hence to the resulting strong restriction $m | q^x - 1$ in the sum over $m$.
}}
 \end{remark}

\noindent
\begin{remark}
{\emph{
It is natural to ask what the best error terms in the asymptotic (\ref{effective-d}) leading to Theorem \ref{thm1} might be. 
For $r \geq 4$,
our methods give rise to a dominant error term 
$ \O_{\psi,K,d} \left(q^{\frac{(r + 2)c_K x}{2 r}} \right)$, which, as $r \rightarrow \infty$, is of the same order of magnitude as the one in (\ref{coto}) and as the best error term with respect to $x$ in the standard Chebotarev Density Theorem \ref{cheb-dm} applied to one (or finitely many) field(s).
When $r = 2$, we obtain
$\O_{\psi,K,d}\left(q^{\frac{5}{6}c_K x}\right)$.
Moreover, by making better use of the properties of Drinfeld modules with a non-trivial endomorphism ring, in \cite{CoSh} we succeed in lowering this error term to $\O_{\psi,K,d}\left(q^{\frac{3}{4}x}\right)$ if $\psi \in \Drin_{A}(k)$ has rank 2 and is such that $\End_{\overline{k}}(\psi)$ is a maximal $A$-order in a field extension of $k$ of degree 2 (and provided $q$ is odd).
It remains to investigate  the true order of magnitude of the error term for such small $r$.
}}
 \end{remark}

\noindent
\begin{remark}
{\emph{
It is  also natural to investigate the positivity of the densities $\delta_{\psi, K}(d)$ 
in Theorem \ref{thm1}.
 The methods required for such a study are of a completely different nature than the ones used in the present paper and, as such,   this study is to be addressed separately.
Nevertheless, we can already exhibit Drinfeld modules  for which some of these densities are positive. For example,  a consequence of part (b) of Theorem 5 in  \cite{CoPa} is that, for $q$ odd and 
 for any rank 2 Drinfeld
$A$-module $\psi$ over $k$  with $\End_{\overline{k}}(\psi)$ the maximal order in an imaginary quadratic extension of $k$, we have
$$
\delta_{\psi, k}(1) \geq \frac{1}{2}.
$$
A consequence of the main result of \cite{Zy} is that, 
 for the rank 2  Drinfeld module $\psi \in \Drin_{A}(k)$ defined by
 $$
 \psi_T = T + \tau - T^{q-1} \tau^2,
 $$
 we have
 $$
 \delta_{\psi, k}(d) = \ds\prod_{\ell \in A^{(1)} \atop{\ell \; \text{prime}}} \left(1 - \frac{1}{\#\GL_2(A/\ell d A)}\right) > 0.
 $$
}}
 \end{remark}

\vspace*{1cm}

\noindent
{\bf{Acknowledgments.}} Some of the work on this paper was done while A.C. Cojocaru  was a member 
at the Institute for Advanced Study  in Princeton, USA,
and a guest researcher at the Max Planck
Institute for Mathematics  in Bonn and the University of G\"{o}ttingen, Germany.
She is grateful to all institutes for excellent work facilities and funding.
In particular, she is grateful for research support from the National Science Foundation under agreements No. DMS-0747724 and No. DMS-0635607,
and from the  European Research Council under Starting Grant 258713.

Both authors are grateful to 
G. B\"{o}ckle, F. Breuer, I. Chen, J. Long-Hoelscher,  M. Papikian, and D. Thakur for helpful discussions related to the background on Drinfeld modules.


\end{document}